\documentclass[final,onefignum,onetabnum]{siamart220329}

\usepackage{amsmath,amsfonts,amsopn,amssymb}
\usepackage[caption=false]{subfig}
\usepackage{multirow,relsize,makecell}

\usepackage{url}
\ifpdf
  \DeclareGraphicsExtensions{.eps,.pdf,.png,.jpg}
\else
  \DeclareGraphicsExtensions{.eps}
\fi
\usepackage{algorithm, algorithmic}
\usepackage{enumitem}

\numberwithin{theorem}{section}
\newsiamremark{assumption}{Assumption}
\newsiamremark{remark}{Remark}
\newsiamremark{example}{Example}
\crefname{assumption}{Assumption}{Assumptions}
\crefname{remark}{Remark}{Remarks}
\crefname{example}{Example}{Examples}

\newcommand\cI{\mathcal{I}}
\newcommand\cN{\mathcal{N}}
\newcommand\cT{\mathcal{T}}
\newcommand\tu{\tilde{u}}
\newcommand\intO{\int_{\Omega}}

\usepackage{color}

\headers{Additive Schwarz for semilinear elliptic}{Jongho Park}

\title{Additive Schwarz methods for semilinear elliptic problems with convex energy functionals: Convergence rate independent of nonlinearity\thanks{Submitted to arXiv
}}

\author{
Jongho Park\thanks{Computer, Electrical and Mathematical Science and Engineering Division, King Abdullah University of Science and Technology~(KAUST), Thuwal 23955, Saudi Arabia
 (\email{jongho.park@kaust.edu.sa}, \url{https://sites.google.com/view/jonghopark}).}
}

\ifpdf
\hypersetup{
  pdftitle={Additive Schwarz methods for semilinear elliptic problems with convex energy functionals},
  pdfauthor={Jongho Park}
}
\fi

\begin{document}

\maketitle

\begin{abstract}
We investigate additive Schwarz methods for semilinear elliptic problems with convex energy functionals, which have wide scientific applications.
A key observation is that the convergence rates of both one- and two-level additive Schwarz methods have bounds independent of the nonlinear term in the problem.
That is, the convergence rates do not deteriorate by the presence of nonlinearity, so that solving a semilinear problem requires no more iterations than a linear problem.
Moreover, the two-level method is scalable in the sense that the convergence rate of the method depends on $H/h$ and $H/\delta$ only, where $h$ and $H$ are the typical diameters of an element and a subdomain, respectively, and $\delta$ measures the overlap among the subdomains.
Numerical results are provided to support our theoretical findings.
\end{abstract}



\section{Introduction}
\label{Sec:Introduction}
The objective of this paper is to conduct an analysis of overlapping domain decomposition methods for solving a specific class of nonlinear elliptic partial differential equations.
Our problem of interest is a semilinear elliptic boundary value problem that accommodates a variational formulation of the following form:
\begin{equation}
\label{model}
    \min_{u \in H_0^1 (\Omega)} \left\{ E(u) := \int_{\Omega} \left( \frac{1}{2} | \nabla u |^2 + \phi (x, u) \right) \,dx \right\},
\end{equation}
where $\Omega \subset \mathbb{R}^d$~($d = 2,3)$ is a bounded polyhedral domain, and $\phi(x,y) \colon \Omega \times \mathbb{R} \rightarrow \mathbb{R}$ is a smooth function that is convex with respect to $y$.
Problems of the form~\eqref{model} include many important problems from diverse fields: the Poisson--Boltzmann equation in electrostatics and biomolecular modeling~\cite{CHX:2007,LZHM:2008}, nonlinear Schr\"{o}dinger equation in optics and the wave theory~\cite{ZM:1974}, Liouville equation in differential geometry~\cite{DFN:1992}, and shape optimization problems in structural mechanics~\cite{HS:2016}.
Mathematical theories of the variational approach for semilinear elliptic problems can be found in, e.g.,~\cite{BS:2011}.

An important breakthrough in numerical solutions for semilinear elliptic problems is the two-grid method~\cite{Xu:1994,Xu:1996}.
This method demonstrated that the nonlinearity inherent in a semilinear elliptic problem can be \textit{coarsened}.
Namely, by discretizing the equation using two finite element spaces~(fine and coarse), it was shown that handling the nonlinearity only on the coarse grid is sufficient to achieve an approximate solution with accuracy comparable to that obtained using the fine grid.
In this context,~\cite{Xu:1996} concluded that
\begin{quote}
``This means that solving a nonlinear equation is not much more difficult than solving one linear equation."
\end{quote}
This paper takes a different perspective on nonlinearity, reaching a similar conclusion: we establish that the nonlinearity of a semilinear elliptic problem can be \textit{localized}.
Specifically, we show that, when employing a domain decomposition method~\cite{TW:2005} with nonlinear local solvers to solve a semilinear elliptic problem, the nonlinear term does not affect the global convergence rate of the algorithm.
That is, solving a semilinear problem requires no more iterations than a linear problem.

In this paper, we study additive Schwarz methods for solving the variational problem~\eqref{model}.
In additive Schwarz methods, a computational domain is partitioned into a number of subdomains, and local problems defined on the subdomains are solved in parallel.
Then the computed local solutions are assembled to form a correction for the global solution.
Hence, they are suitable for massively parallel computation, and have been applied for numerical solutions of various partial differential equations~\cite{DCPS:2012,DW:2009,HKKR:2019,HKKRW:2022,Oh:2013}.
In particular, convergence theories for additive Schwarz methods for convex optimization problems were developed in~\cite{Badea:2010,Park:2020,Park:2022a,TX:2002}, and then applied to several nonlinear problems, including the $p$-Laplacian equation~\cite{LP:2022} and elliptic optimal control~\cite{Park:2023}.

To analyze additive Schwarz methods for the smooth convex optimization problem~\eqref{model}, adopting the smooth theory established in~\cite{TX:2002} emerges as a natural approach.
However, in this case, the nonlinear term $\phi (x, u)$ in~\eqref{model} may result in poor convergence rates.
For example, in the nonlinear Poisson--Boltzmann equation, $\phi (x, u)$ is given by a hyperbolic function with rapid exponential nonlinearity~\cite{CHX:2007}, so that it can potentially yield a very large condition number of the entire energy functional.
To address this issue, we present a novel perspective in this paper; we leverage the nonsmooth theory developed in~\cite{BK:2012,Park:2020} to analyze the smooth problem~\eqref{model}.
More precisely, we analyze the methods by classifying the quadratic and $\phi$-terms in the energy functional as smooth and nonsmooth, respectively.
We construct stable decompositions suitable for the nonsmooth framework in terms of the nonlinear positivity-preserving coarse interpolation operator introduced in~\cite{Badea:2006,Tai:2005}.
Finally, we derive estimates for the convergence rates of the methods, and importantly, these estimates remain independent of the $\phi$-term~(see \cref{Thm:1L,Thm:2L}).
That is, we establish that the convergence rates of additive Schwarz methods for~\eqref{model} are uniformly bounded regardless of the specific instance of the nonlinear term $\phi$.
This result is particularly surprising, as it contradicts our common intuition that nonlinearity typically significantly influences the convergence performance of iterative methods for nonlinear problems.
In addition, \cref{Thm:2L} states that the convergence rate of the two-level additive Schwarz method depends on $H/h$ and $H/\delta$ only, where $h$ and $H$ are the typical diameters of an element and a subdomain, respectively.
This implies that the two-level method is scalable; even when the element size $h$ is very small, we do not need a large number of iterations to obtain a solution with a desired level of accuracy if we have a sufficient number of subdomains.

The rest of this paper is organized as follows.
In \cref{Sec:Model}, we characterize semilinear elliptic problems that accommodate variational formulations of the form~\eqref{model}, and introduce their finite element discretizations.
In \cref{Sec:ASM}, we present the abstract framework of additive Schwarz methods for convex optimization introduced in~\cite{Park:2020} in a way suitable for our problem of interest.
In \cref{Sec:DD}, we present convergence analyses of one- and two-level additive Schwarz methods for semilinear elliptic problems based on overlapping domain decomposition.
In \cref{Sec:Numerical}, we present numerical results for various semilinear elliptic problems that verify our theoretical findings.
Finally, we conclude the paper with discussions in \cref{Sec:Conclusion}.

\section{Model problem}
\label{Sec:Model}
In this section, we briefly introduce semilinear elliptic problems and their variational formulations.
Based on the variational formulations, we derive finite element discretizations and their error estimates.

We first introduce notation that will be used throughout the paper.
Let $V$ be a Hilbert space equipped with the norm $\| \cdot \|_V$.
The topological dual space of $V$ is denoted by $V^*$, and $\langle \cdot, \cdot \rangle_{V^* \times V}$ denotes the duality pairing of $V$, i.e.,
\begin{equation*}
\langle f, u \rangle_{V^* \times V} = f(u),
\quad u \in V, \text{ } f \in V^*.
\end{equation*}
We may omit the subscripts if there is no ambiguity.
For a smooth convex functional $F \colon V \rightarrow \mathbb{R}$ and $u \in V$, let $F' (u) \in V^*$ denote the G\^{a}teaux derivative of $F$ at $u$, i.e.,
\begin{equation*}
 \langle F'(u), v \rangle = \lim_{t \rightarrow 0} \frac{F(u + tv) - F(u)}{t},
 \quad v \in V.
\end{equation*}
The Bregman distance of $F$ is defined by
\begin{equation}
\label{Bregman}
D_F (u, v) = F(u) - F(v) - \langle F'(v), u-v \rangle,
\quad u, v \in V.
\end{equation}

\subsection{Semilinear elliptic problems}
A general second-order semilinear elliptic problem on $\Omega$ with the homogeneous Dirichlet boundary condition is written as 
\begin{equation} \begin{split}
\label{semilinear}
- \Delta u + f(x, u) = 0 \quad &\text{ in } \Omega, \\
u = 0 \quad &\text{ on } \partial \Omega,
\end{split} \end{equation}
where $f(x,y) \colon \Omega \times \mathbb{R} \rightarrow \mathbb{R}$ is a sufficiently regular function.
It is well-known that, under some mild assumptions on $f$, a solution of~\eqref{semilinear} can be characterized as a critical point of an energy functional~\cite{BS:2011}.
In particular, if we have
\begin{equation}
\label{semilinear_convex}
\intO f(x, u) v\,dx = \left< \left( \intO \phi (x, \cdot (x)) \,dx \right)' (u), v \right>, \quad u, v \in H_0^1 (\Omega),
\end{equation}
then~\eqref{semilinear} admits~\eqref{model} as its variational formulation.
Hence, under~\eqref{semilinear_convex}, we may consider the convex optimization problem~\eqref{model} in order to obtain the solution of the semilinear elliptic problem~\eqref{semilinear}.
The following proposition summarizes a sufficient condition for~\eqref{semilinear_convex}.

\begin{proposition}
\label{Prop:semilinear_convex}
In~\eqref{model} and~\eqref{semilinear}, suppose that the following hold:
\begin{enumerate}[label=\emph{(\roman*)}]
\item $\phi (\cdot, u(\cdot)) \in L^1 (\Omega)$ and $f( \cdot, u(\cdot)) \in L^2 (\Omega)$ for any $u \in H_0^1 (\Omega)$,
\item $\phi (x,y)$ is convex with respect to $y$,
\item $f(x,y) = \frac{\partial \phi}{\partial y} (x,y)$.
\end{enumerate}
Then the map $u \mapsto \intO \phi (x, u(x)) \,dx$ is G\^{a}teaux differentiable and~\eqref{semilinear_convex} holds.
Consequently,~\eqref{semilinear} admits the convex variational formulation~\eqref{model}.
\end{proposition}
\begin{proof}
See \cref{App:Semilinear_convex}.
\end{proof}

We present two simple applications of \cref{Prop:semilinear_convex} in the following.

\begin{example}
\label{Ex:monomial}
Set $f(x,y) = \alpha |y|^{m-2} y - g(x)$ in~\eqref{semilinear} for some $\alpha \geq 0$ and $m \in \mathbb{Z}_{\geq 2}$.
Then we have
\begin{equation*}
- \Delta u + \alpha |u|^{m-2} u = g.
\end{equation*}
Note that this problem was considered in, e.g.,~\cite{Xu:1994}.
In this case,~\eqref{semilinear_convex} holds with $\phi (x,y) = \frac{\alpha}{m} |y|^m - g(x) y$.
\end{example}

\begin{example}
\label{Ex:PB}
Set $f(x,y) = \sinh y - g(x)$ in~\eqref{semilinear} for some $g \in L^2 (\Omega)$.
Then we have the nonlinear Poisson--Boltzmann equation~\cite{CHX:2007,LZHM:2008} as follows:
\begin{equation*}
- \Delta u + \sinh u = g.
\end{equation*}
In this case,~\eqref{semilinear_convex} holds with $\phi (x,y) = \cosh y - g(x) y$.
\end{example}

Assuming~\eqref{semilinear_convex}, the minimization problem~\eqref{model} is equivalent to the following: Find $u \in H_0^1 (\Omega)$ such that
\begin{equation}
\label{model_weak}
    a(u,v) + b(u,v) = 0, \quad v \in H_0^1 (\Omega),
\end{equation}
where $a(\cdot, \cdot) \colon H_0^1 (\Omega) \times H_0^1 (\Omega) \rightarrow \mathbb{R}$ and $b(\cdot, \cdot) \colon H_0^1 (\Omega) \times H_0^1 (\Omega) \rightarrow \mathbb{R}$ are given by
\begin{equation*}
    a(u,v) = \intO \nabla u \cdot \nabla v \,dx, \quad
    b(u,v) = \intO f(x, u) v \,dx, \quad
    u,v \in H_0^1 (\Omega).
\end{equation*}
Note that $b$ is not linear with respect to its first argument.

\subsection{Finite element discretization}
Now, we consider a finite element discretization of the variational formulation~\eqref{model}.
We assume that the domain $\Omega$ admits a quasi-uniform triangulation $\cT_h$ with $h$ the characteristic element diameter.
Let $S_h (\Omega)$ be the collection of continuous and piecewise linear functions on $\cT_h$ vanishing on $\partial \Omega$.
The nodal interpolation operator $I_h$ onto the finite element space $S_h (\Omega)$ is well-defined for functions in $C^0 (\overline{\Omega})$.
We introduce the following finite element discretization of~\eqref{model} defined on the finite element space $S_h (\Omega)$:
\begin{equation}
\label{model_FEM}
    \min_{u \in S_h (\Omega)} \left\{ E_h(u) := \int_{\Omega} \left( \frac{1}{2} | \nabla u |^2 + I_h \left( \phi (\cdot, u) \right) \right) \,dx \right\}.
\end{equation}
In~\eqref{model_FEM}, compared to the continuous problem~\eqref{model}, the solution space is replaced with the finite element space $S_h (\Omega)$ and the $\phi$-term is replaced with its approximation using the trapezoidal rule on the simplicial mesh $\cT_h$.
By applying the trapezoidal rule to the $\phi$-term, the integral of this term becomes separable with respect to each nodal degree of freedom of $S_h (\Omega)$.
To be more precise, we have
\begin{equation}
\begin{split}
    \label{separable}
    \intO I_h \left( \phi (\cdot, u) \right) \,dx
    &= \sum_{T \in \cT_h} \int_T I_h (\phi (\cdot, u)) \,dx \\
    &= \sum_{T \in \cT_h} \sum_{x \in \cN_T} \frac{|T|}{3} \phi (x, u(x)) \\
    &= \sum_{x \in \cN_h} \sum_{T \in \cT_x} \frac{|T|}{3} \phi (x, u(x)) \\
    &= \sum_{x \in \cN_h} \phi_h^x (u(x)),
\end{split}
\end{equation}
where $\cN_T$ is the set of vertices of an element $T$, 
$\cN_h$ denotes the set of vertices in the triangulation $\cT_h$,
$\cT_x$ is the collection of all elements that share a vertex $x$,
and $\phi_h^x \colon \mathbb{R} \rightarrow \mathbb{R}$ is a function defined by
\begin{equation*}
    \phi_h^x (y) = \sum_{T \in \cT_x} \frac{|T|}{3} \phi (x, y), \quad y \in \mathbb{R},
\end{equation*}
for each $x \in \cN_h$ and $h > 0$.
Since $\phi (x,y)$ is convex with respect to $y$, we deduce that $\phi_h^x$ is convex.
The separation property expressed in~\eqref{separable} will play a critical role in establishing convergence rates of additive Schwarz methods independent of the $\phi$-term.

\begin{remark}
\label{Rem:higher}
The separation property of the $\phi$-term in~\eqref{model_FEM} can be established for higher-order finite elements, assuming that the numerical integration on each element is computed by a convex combination of nodal values.
Consider the case when $S_h (\Omega)$ is given by a higher-order Lagrangian finite element space, and $I_h$ is the nodal interpolation operator onto $S_h (\Omega)$.
Since the numerical integration of $I_h (\phi (\cdot, u))$ over an element $T \in \mathcal{T}_h$ is expressed as a positive linear combination of the nodal values of $\phi (\cdot, u)$, one can derive the separation property for the $\phi$-term similar to~\eqref{separable}.
\end{remark}

It is easy to verify that~\eqref{model_FEM} admits a unique solution, say $u_h \in S_h (\Omega)$.
An error estimate for $u_h$ compared with the continuous solution $u \in H_0^1 (\Omega)$ of~\eqref{model} is summarized in \cref{Thm:FEM}.

\begin{theorem}
\label{Thm:FEM}
Let $u \in H_0^1 (\Omega)$ and $u_h \in S_h (\Omega)$ be the solutions of~\eqref{model} and~\eqref{model_FEM}, respectively.
In addition, let $\tu_h \in S_h (\Omega)$ be a solution of the minimization problem
\begin{equation}
\label{model_FEM_alt}
\min_{u \in S_h (\Omega)} E(u).
\end{equation}
Suppose that the assumptions in \cref{Prop:semilinear_convex} and the following hold:
\begin{enumerate}[label=\emph{(\roman*)}]
\item $u \in H^2 (\Omega)$,
\item $v \mapsto f(\cdot, v)$ is Fr\'{e}chet differentiable.
\end{enumerate}
Then there exists a positive constant $C$ independent of $h$ such that
\begin{equation}
\label{Thm1:FEM}
\| u - u_h \|_{H^1 (\Omega)}^2 \leq C h^2 \bigg( 1 + \max_{T \in \cT_h} \| \phi (\cdot, u_h) \|_{W^{2, \infty} (T)}
+ \max_{T \in \cT_h} \| \phi (\cdot, \tu_h) \|_{W^{2, \infty} (T)} \bigg).
\end{equation}
\end{theorem}
\begin{proof}
See \cref{App:FEM}.
\end{proof}

\begin{remark}
\label{Rem:FEM}
The right-hand side of~\eqref{Thm1:FEM} can be uniformly bounded with respect to $h$ if $\phi$ is sufficiently smooth and $u_h$ and $\tilde{u}_h$ satisfy certain $L^{\infty}$-type stability conditions.
It is worth noting that, as shown in~\cite{KJ:1990}, piecewise linear finite element discretizations of semilinear elliptic problems exhibit $L^{\infty}$-stability under specific assumptions on the mesh $\cT_h$. For a particular instance of the nonlinear Poisson--Boltzmann equation, refer to~\cite{CHX:2007}.
\end{remark}

We conclude this section by mentioning that, by invoking~\eqref{separable}, one can deduce that~\eqref{model_FEM} is equivalent to find $u \in S_h (\Omega)$ such that
\begin{equation}
\label{model_FEM_weak}
    a(u,v) + b_h (u, v) = 0, \quad v \in S_h (\Omega),
\end{equation}
where $b_h (\cdot, \cdot) \colon S_h (\Omega) \times S_h (\Omega) \rightarrow \mathbb{R}$ is given by
\begin{equation*}
    b_h (u, v) = \intO I_h \left( f (\cdot, u) v \right) \,dx,
    \quad u, v \in S_h (\Omega).
\end{equation*}

\section{Additive Schwarz method}
\label{Sec:ASM}
In this section, we provide a brief summary of the abstract framework of additive Schwarz methods for convex optimization, which was originally proposed in~\cite{Park:2020}, applied to the finite element discretization~\eqref{model_FEM}.
Subsequently, we present a convergence theorem for additive Schwarz methods for~\eqref{model_FEM}, providing a sharper result than the one presented in~\cite{Park:2020}.

We assume that the solution space $V = S_h (\Omega)$ of~\eqref{model_FEM} admits a space decomposition of the form
\begin{equation}
\label{space_decomp}
V = \sum_{k=1}^N R_k^* V_k,
\end{equation}
where $V_k$, $1 \leq k \leq N$, is a finite-dimensional space and $R_k^* \colon V_k \rightarrow V$ is an injective linear operator.
The abstract additive Schwarz method for solving~\eqref{model_FEM} under the space decomposition~\eqref{space_decomp} is described in \cref{Alg:ASM}.
This algorithm has been previously studied in some existing works, e.g.,~\cite{Park:2020,Park:2021,TX:2002}.

\begin{algorithm}
\caption{Additive Schwarz method for~\cref{model_FEM}}
\begin{algorithmic}[]
\label{Alg:ASM}
\STATE Let $u^{(0)} \in V$ and $\tau \in (0, \tau_0]$.
\FOR{$n=0,1,2,\dots$}
\item \vspace{-0.5cm} \begin{equation*}
\begin{split}
w_k^{(n+1)} &= \operatornamewithlimits{\arg\min}_{w_k \in V_k} E_h (u^{(n)} + R_k^* w_k ), \quad 1 \leq k \leq N \\
u^{(n+1)} &= u^{(n)} + \tau \sum_{k=1}^N R_k^* w_k^{(n+1)}
\end{split}
\end{equation*} \vspace{-0.4cm}
\ENDFOR
\end{algorithmic}
\end{algorithm}

\Cref{Alg:ASM} consists of local problems of the form
\begin{equation}
\label{local}
\min_{w_k \in V_k} E_h (v + R_k^* w_k),
\end{equation}
where $v \in V$ and $1 \leq k \leq N$.
It is clear that, under the condition~\eqref{semilinear_convex},~\eqref{local} is equivalent to find $w_k \in V_k$ such that
\begin{equation*}
a_k (w_k, v_k) + b_k (v; w_k, v_k) = - a(v, R_k^* v_k), \quad v_k \in V_k,
\end{equation*}
where $a_k(\cdot, \cdot, \cdot) \colon V_k \times V_k \rightarrow \mathbb{R}$ and  $b_k(\cdot; \cdot, \cdot) \colon V \times V_k \times V_k \rightarrow \mathbb{R}$ are given by
\begin{equation*}
\begin{array}{l}
a_k (w_k, v_k) = a(R_k^* w_k, R_k^* v_k), \\
b_k (v; w_k, v_k) = b_h (v + R_k^* w_k, R_k^* v_k),
\end{array}
\quad 
v \in V, \text{ } v_k, w_k \in V_k.
\end{equation*}
That is, each local problem of \cref{Alg:ASM} has the same form as the full problem~\eqref{model_FEM_weak}.

The maximum step size $\tau_0$ in \cref{Alg:ASM} will be described in \cref{Ass:convex}.
Alternatively, the step size $\tau$ can be determined by a full backtracking scheme proposed in~\cite{Park:2022a}, which leads to faster convergence.
For simplicity, we deal with the case of constant $\tau$; extending to variable step sizes is straightforward.

Here, we summarize the convergence theory of \cref{Alg:ASM} introduced in~\cite{Park:2020}.
We suppose that the energy functional $E_h$ in~\eqref{model_FEM} is decomposed as
\begin{equation}
    \label{composite}
    E_h (u) = F_h (u) + G_h (u),
\end{equation}
where $F_h \colon V \rightarrow \mathbb{R}$ is a G\^{a}teaux differentiable convex functional and $G_h \colon V \rightarrow \mathbb{R}$ is a convex functional.
Although the trivial decomposition $F_h = E_h$ and $G_h = 0$ is possible, we will use a different decomposition~(see~\eqref{FG}) in our analysis to obtain a convergence rate that is independent of the $\phi$-term in $E_h$.

The stable decomposition assumption, which was introduced in~\cite[Assumption~4.1]{Park:2020}, associated with the space decomposition~\eqref{space_decomp} and the energy decomposition~\eqref{composite} is stated in \cref{Ass:stable} in a form suitable for our purpose.

\begin{assumption}[stable decomposition]
    \label{Ass:stable}
    There exists a positive constant $C_0$ such that the following holds: for any $u,v \in V$, there exist $w_k \in V_k$, $1 \leq k \leq N$, such that
    \begin{subequations}
    \begin{align}
    \label{Ass1:stable}
    u-v &= \sum_{k=1}^N R_k^* w_k, \\
    \label{Ass2:stable}
    \sum_{k=1}^N D_{F_h} (v + R_k^* w_k, v) &\leq \frac{C_0^2}{2} | u - v |_{H^1 (\Omega)}^2,
    \end{align}
    and
    \begin{equation}
        \label{Ass3:stable}
    \sum_{k=1}^N G_h (v+ R_k^* w_k) \leq G_h (u) + (N-1) G_h (v).
    \end{equation}
    \end{subequations}
\end{assumption}

The strengthened convexity assumption~\cite[Assumption~4.2]{Park:2020}, which directly generalizes the strengthened Cauchy--Schwarz inequality for linear problems~\cite[Assumption~2.2]{TW:2005}, is presented in the following.

\begin{assumption}[strengthened convexity]
    \label{Ass:convex}
    There exists a constant $\tau_0 \in (0, 1]$ such that, for any $v \in V$, $w_k \in V_k$, $1 \leq k \leq N$, and $\tau \in (0, \tau_0]$, we have
    \begin{equation*}
        (1 - \tau N) E_h(v) + \tau \sum_{k=1}^N E_h (v + R_k^* w_k) \geq E_h \left( v + \tau \sum_{k=1}^N R_k^* w_k \right).
    \end{equation*}
\end{assumption}

We note that the abstract convergence theory presented in~\cite{Park:2020} has two other assumptions: local stability~\cite[Assumption~4.3]{Park:2020} and sharpness of the energy functional~\cite[Assumption~3.4]{Park:2020}.
On the one hand, the local stability assumption is trivially satisfied in our case because \cref{Alg:ASM} uses exact local solvers~\eqref{local}.
On the other hand, it is readily verified that $E_h$ is 1-strongly convex~\cite{Nesterov:2018} with respect to the seminorm $| \cdot |_{H^1 (\Omega)}$~(which is indeed a norm by the Poincar\'{e}--Friedrichs inequality~\cite{BS:2008}):
\begin{equation}
\label{E_strong}
E_h (tu + (1-t)v) \leq t E(u) + (1-t) E(v) - \frac{t(1-t)}{2} |u - v |_{H^1 (\Omega)}^2, \quad
t \in [0, 1], \text{ }
u , v \in V.
\end{equation}
From~\eqref{E_strong}, we can deduce the following inequality without major difficulty~(cf.~\cite[Proposition~3.5]{Park:2020}):
\begin{equation}
\label{E_sharp}
    E_h(u) - E_h(u_h) \geq \frac{1}{2} | u - u_h |_{H^1 (\Omega)}^2, \quad u \in V,
\end{equation}
which implies that the sharpness assumption holds.
Hence, in our case, it suffices to consider \cref{Ass:stable,Ass:convex} only.

Under \cref{Ass:stable,Ass:convex}, the application of~\cite[Theorem~4.8]{Park:2020} to \cref{Alg:ASM} deduces that \cref{Alg:ASM} exhibits linear convergence,
in which the convergence rate depends on $C_0$ and $\tau_0$ only; see \cref{Rem:Park:2020}.
Meanwhile, the convergence rate can be further improved by incorporating the strong convexity of $E_h$ into the analysis.
Note that the 1-strong convexity of $E_h$ implies $\mu$-strong convexity of $F_h$ for some $\mu \in [0, 1]$.
The improved convergence result of \cref{Alg:ASM} is summarized in \cref{Thm:conv}; a proof will be provided in \cref{App:conv}.

\begin{theorem}
    \label{Thm:conv}
    Suppose that \cref{Ass:stable,Ass:convex} hold.
    In \cref{Alg:ASM}, if $\tau \in (0, \tau_0]$, then we have
    \begin{equation*}
        \frac{E_h (u^{(n+1)}) - E_h (u_h)}{E_h (u^{(n)}) - E_h (u_h)} \leq 1 - \tau \min \left\{ 1, \frac{1}{C_0^2 + 1 - \mu} \right\},
        \quad n \geq 0,
    \end{equation*}
where $\mu \in [0, 1]$ is the strong convexity parameter of $F_h$ with respect to the seminorm $| \cdot |_{H^1 (\Omega)}$.
\end{theorem}

Thanks to \cref{Thm:conv}, it suffices to estimate $C_0$ and $\tau_0$ in \cref{Ass:stable,Ass:convex}, respectively, as well as $\mu$ in order for a complete analysis for the convergence rate of \cref{Alg:ASM}.
In overlapping domain decomposition methods, the constant $\tau_0$ is determined by the usual coloring technique~(see, e.g.,~\cite[Section~5.1]{Park:2020}), so that it is independent of domain decomposition.
In \cref{Sec:DD}, we will estimate the constant $C_0$ in \cref{Ass:stable} corresponding to one- and two-level overlapping domain decomposition settings.
Furthermore, we will use a particular decomposition of the form~\eqref{composite} such that $\mu = 1$.

\section{Overlapping domain decomposition methods}
\label{Sec:DD}
In this section, we present convergence analyses of one- and two-level additive Schwarz methods for the finite element discretization~\eqref{model_FEM} based on an overlapping domain decomposition setting.
A special feature of the convergence analysis presented in this section is that the convergence rates of the methods are independent of the $\phi$-term in the energy functional $E_h$.
That is, the convergence rate does not deteriorate by the nonlinearity of the semilinear problem~\eqref{model_FEM_weak}.

\subsection{One-level method}
We assume that the domain $\Omega$ admits a coarse triangulation $\cT_H$ such that $\cT_h$ is a refinement of $\cT_H$, where $H$ stands for the characteristic element diameter of $\cT_H$.
A finite element space $S_H (\Omega)$ on the coarse triangulation $\cT_H$ is defined in the same manner as $S_h (\Omega)$.
Let $\{ \Omega_k \}_{k=1}^N$ be a nonoverlapping domain decomposition of $\Omega$ such that each $\Omega_k$ is the union of several coarse elements in $\cT_H$.
The number of coarse elements consisting each $\Omega_k$, $1 \leq k \leq N$, is assumed to be uniformly bounded.
Each subdomain $\Omega_k$ is enlarged by a layer of fine elements with width $\delta$ to form an extended subdomain $\Omega_k'$.
Then $\{ \Omega_k' \}_{k=1}^N$ forms an overlapping domain decomposition of $\Omega$.
In each $\Omega_k'$, we define $S_h (\Omega_k')$ as the piecewise linear finite element space on $\cT_h |_{\Omega_k'}$ with the homogeneous Dirichlet boundary condition.
As in~\cite[equation~(3.7)]{TW:2005}, we construct a piecewise linear partition of unity $\{ \theta_k \}_{k=1}^N$ for $\Omega$ subordinate to the covering $\{ \Omega_k' \}_{k=1}^N$ that satisfies
\begin{subequations}
\label{pou}
\begin{align}
    \label{pou1}
    \theta_k = 0 \quad \text{ in } \Omega \setminus \Omega_k', \\
    \label{pou2}
    \sum_{k=1}^N \theta_k = 1 \text{ in } \overline{\Omega}, \\
    \label{pou3}
    | \theta_k |_{W^{1,\infty} (\Omega_k')} \lesssim \frac{1}{\delta}, \quad 1 \leq k \leq N,
\end{align}
\end{subequations}
where the notation $A \lesssim B$ means that there exists a positive constant $C$, which is independent of $h$, $H$, and $\delta$, such that $A \leq C B$.
In what follows, we also write $A \approx B$ if $A \lesssim B$ and $B \lesssim A$.

In~\eqref{space_decomp}, we set
\begin{equation}
\label{1L}
V_k = S_h (\Omega_k'), \quad 1 \leq k \leq N,
\end{equation}
and $R_k^* \colon V_k \rightarrow V$ as the natural extension-by-zero operator from $S_h (\Omega_k')$ to $S_h (\Omega)$.
Then \cref{Alg:ASM} becomes the one-level additive Schwarz method for~\eqref{model_FEM}.
Using the coloring technique, one can readily verify that \cref{Ass:convex} is satisfied with $\tau_0 = 1/N_c$ under the setting~\eqref{1L}, where $N_c \leq 4$ if $d = 2$ and $N_c \leq 8$ if $d = 3$~\cite{Park:2020,TX:2002}.
Hence, it is enough to prove \cref{Ass:stable} in order to complete the convergence analysis of the one-level method.

To verify \cref{Ass:stable}, we first have to determine the functionals $F_h$ and $G_h$ in~\eqref{composite}.
In existing analyses of additive Schwarz methods~\cite{Park:2020,Park:2023}, we usually set $F_h$ and $G_h$ as the smooth and nonsmooth parts of the entire energy functional $E_h$, respectively.
From this perspective, it might seem natural to set $F_h = E_h$ and $G_h = 0$ as $E_h$ is smooth.
However, we adopt a different strategy here; we define $F_h$ and $G_h$ as the linear and nonlinear parts of $E_h$ respectively, as follows:
\begin{equation}
    \label{FG}
    F_h(u) = \int_{\Omega} \frac{1}{2} |\nabla u|^2  \,dx , \quad
    G_h (u) = \int_{\Omega} I_h \left( \phi (\cdot, u) \right) \,dx.
\end{equation}
It is clear that $F_h$ is 1-strongly convex with respect to the seminorm $| \cdot |_{H^1 (\Omega)}$, i.e., we have $\mu = 1$.
The Bregman distance~\eqref{Bregman} of $F_h$ is given by
\begin{equation}
    \label{F_distance}
    D_{F_h} (u,v) = \intO \frac{1}{2} | \nabla (u-v) |^2 \,dx, \quad u,v \in V.
\end{equation}
Moreover,~\eqref{separable} implies that $G_h$ is expressed as
\begin{equation}
    \label{G_pointwise}
    G_h (u) = \sum_{x \in \cN_h} \phi_h^x (u(x)).
\end{equation}

The following theorem states that the one-level additive Schwarz method for~\eqref{model_FEM} satisfies \cref{Ass:stable} in a way that the constant $C_0$ is independent of the $\phi$-term in $E_h$.

\begin{theorem}
    \label{Thm:1L}
    In the one-level additive Schwarz method described in~\eqref{space_decomp} and~\eqref{1L}, \cref{Ass:stable} holds with
    \begin{equation*}
        C_0^2 \approx \frac{1}{H \delta}.
    \end{equation*}
\end{theorem}
\begin{proof}
    Take any $u,v \in V$ and let $w = u - v$.
    We define $w_k \in V_k$, $1 \leq k \leq N$, as
    \begin{equation*}
        w_k = I_h (\theta_k w ),
    \end{equation*}
    where $\{ \theta_k \}_{k=1}^N$ is the partition of unity given in~\eqref{pou}.
    We readily observe that~\eqref{Ass1:stable} holds.
    Since $0 \leq \theta_k (x) \leq 1$ for each $k$ and $x \in \cN_h$, invoking~\eqref{G_pointwise} and the convexity of $\phi_h^x$ yields
    \begin{equation*} \begin{split}
        \sum_{k=1}^N G_h (v + R_k^* w_k)
        &= \sum_{x \in \cN_h} \sum_{k=1}^N \phi_h^x \left( v(x) + \theta_k (x) w(x) \right) \\
        &\leq \sum_{x \in \cN_h} \sum_{k=1}^N \left[ \theta_k (x) \phi_h^x \left( v(x) + w(x) \right) + (1 - \theta_k (x) ) \phi_h^x \left( v(x) \right) \right] \\
        &= G_h (u) + (N-1) G_h (v),
    \end{split} \end{equation*}
    which verifies~\eqref{Ass3:stable}.
    It is worth mentioning that a similar argument was presented in~\cite[Proposition~5.1]{Badea:2010}.
    
    While estimating $C_0^2$ in~\eqref{Ass2:stable} can be done by a standard argument of overlapping Schwarz methods~(see, e.g.,~\cite{Park:2023,TW:2005}), we provide a detailed derivation for the sake of completeness.
    By~\eqref{F_distance} and the $H^1$-stability of the nodal interpolation operator $I_h$ presented in~\cite[Lemma~3.9]{TW:2005}, we have
    \begin{equation}
        \label{Thm1:1L}
        D_{F_h}(v + R_k^* w_k, v) 
        \lesssim \intO |\nabla (R_k^* w_k)|^2 \,dx
        = |I_h (\theta_k w)|_{H^1 (\Omega)}^2
        \lesssim |\theta_k w|_{H^1 (\Omega)}^2.
    \end{equation}
    Invoking~\eqref{pou} yields
    \begin{equation}
        \label{Thm2:1L}
        | \theta_k w |_{H^1 (\Omega)}^2
        \lesssim \intO | \theta_k \nabla w |^2 \,dx + \intO |w \nabla \theta_k |^2 \,dx
        \lesssim | w |_{H^1 (\Omega)}^2 + \frac{1}{ \delta^2} \int_{\Omega_{k, \delta}} |w|^2 \,dx,
    \end{equation}
    where $\Omega_{k, \delta}$ is the support of $\nabla \theta_k$.
    Since $\Omega_{k, \delta}$ is a strip of width $\approx \delta$ along the boundary $\partial \Omega_k'$ of the subdomain $\Omega_k'$, we can use the trace theorem-type argument introduced in~\cite[Lemma~3.10]{TW:2005}~(see also~\cite{DW:1994}) to derive the following:
    \begin{equation}
        \label{Thm3:1L}
        \frac{1}{\delta^2} \int_{\Omega_{k, \delta}} |w|^2 \,dx
        \lesssim \left( 1 + \frac{H}{\delta} \right) | w |_{H^1 (\Omega)}^2 + \frac{1}{H \delta} \| w \|_{L^2 (\Omega)}^2
        \lesssim \frac{1}{H \delta} | w |_{H^1 (\Omega)}^2,
    \end{equation}
    where the last inequality is due to the Poincar\'{e}--Friedrichs inequality~\cite{BS:2008}.
    By~\eqref{Thm1:1L}--\eqref{Thm3:1L}, we deduce that~\eqref{Ass2:stable} holds with $C_0^2 \approx 1/H\delta$. 
\end{proof}

Thanks to \cref{Thm:1L}, we guarantee that the convergence rate of the one-level additive Schwarz method for~\eqref{model_FEM} has a bound independent of the nonlinear term in~\eqref{model_FEM}.
That is, the convergence rate remains stable even if nonlinearity of the problem is very strong, e.g., exponential nonlinearity in the nonlinear Poisson--Boltzmann equation~\cite{CHX:2007}.
Meanwhile, since $C_0^2$ in \cref{Thm:1L} depends on $1/H \delta$, the method is not scalable in the sense that increasing the number of subdomains leads to a higher number of iterations required to attain a desired level of accuracy.
To achieve scalability, it becomes necessary to design a two level method, which utilizes a suitable coarse space that effectively corrects the low-frequency error of the solution~\cite{TW:2005}.

\subsection{Two-level method}
In the two-level method, we use a space decomposition
\begin{equation}
\label{space_decomp_2L}
V = R_0^* V_0 + \sum_{k=1}^N R_k^* V_k,
\end{equation}
where $V_0$ is a finite-dimensional space that plays a role of the coarse space, and $R_0^* \colon V_0 \rightarrow V$ is an injective linear operator.
The two-level additive Schwarz method based on~\eqref{space_decomp_2L} is the same as \cref{Alg:ASM} except that the index $k$ runs from $0$ to $N$, so that we do not present it separately.
In~\eqref{space_decomp_2L}, we set $V_0$ as
\begin{equation}
    \label{2L}
    V_0 = S_H (\Omega),
\end{equation}
i.e., the piecewise linear finite element space on the coarse triangulation $\cT_H$.
In addition, we set $V_k$, $1 \leq k \leq N$, as in~\eqref{1L}.
As $S_H (\Omega) \subset S_h (\Omega)$, we set $R_0^*$ as the natural embedding from $S_H (\Omega)$ to $S_h (\Omega)$.

Let $\{ x^i \}_{i \in \cI_H}$ be the collection of all interior vertices of $\cT_H$.
For each $i \in \cI_H$, we define a region $\omega_i \subset \Omega$ as the union of the coarse elements having $x^i$ as its vertex:
\begin{equation*}
    \overline{\omega}_i = \bigcup_{T \in \cT_H, x^i \in \partial T} \overline{T}.
\end{equation*}
Let $\phi^i$ be the coarse nodal basis function associated with the vertex $x^i$.
We recall that the following nonlinear coarse interpolation operator $J_H \colon S_h (\Omega) \rightarrow S_H (\Omega)$ was introduced in~\cite{Badea:2006,Tai:2005}:
\begin{equation}
    \label{J_H}
    J_H u = \sum_{i \in \cI_H} \left( \min_{\overline{\omega}_i} \max \{ u, 0 \} - \min_{\overline{\omega}_i} \max \{ -u, 0 \} \right)\phi^i.
\end{equation}
A key property of $J_H$ is that it is positivity-preserving, i.e., it satisfies
\begin{equation}
\label{positivity}
\begin{cases}
    0 \leq J_H w \leq w, &\quad \text{ if } w > 0, \\
    w \leq J_H w \leq 0, &\quad \text{ if } w < 0, \\
    J_H w = 0, &\quad \text{ if } w = 0.
\end{cases}
\end{equation}
The error and stability estimates of $J_H$ are summarized in \cref{Lem:J_H}.
We note that the proof of \cref{Lem:J_H} given in~\cite{Tai:2005}~(see also~\cite{Tai:2003}) relies on the discrete Sobolev inequality of piecewise linear finite element functions~\cite[Lemma~2.3]{BX:1991}.

\begin{lemma}
    \label{Lem:J_H}
    The operator $J_H \colon S_h (\Omega) \rightarrow S_H (\Omega)$ defined in~\eqref{J_H} satisfies
    \begin{equation*}
        \| u - J_H u \|_{L^2 (\Omega)} + H | J_H u |_{H^1 (\Omega)}
        \lesssim C_d (H, h) H | u |_{H^1 (\Omega)}, \quad u \in S_h (\Omega),
    \end{equation*}
    where $C_d (H,h)$ is given by
    \begin{equation}
    \label{C_d}
        C_d (H, h) = \begin{cases}
            \left( 1 + \log \dfrac{H}{h} \right)^{\frac{1}{2}}, & \quad \text{ if } d = 2, \\
            \left( \dfrac{H}{h} \right)^{\frac{1}{2}}, & \quad \text{ if } d = 3.
        \end{cases}
    \end{equation}
\end{lemma}

Using \cref{Lem:J_H}, we are able to prove that the two-level additive Schwarz method for~\eqref{model_FEM} satisfies \cref{Ass:stable} as follows.

\begin{theorem}
    \label{Thm:2L}
    In the two-level additive Schwarz method described in~\eqref{1L},~\eqref{space_decomp_2L}, and~\eqref{2L}, \cref{Ass:stable} holds with
    \begin{equation*}
        C_0^2 \approx C_d (H, h)^2 \left( 1 + \frac{H}{\delta} \right),
    \end{equation*}
    where $C_d$ was given in~\eqref{C_d}.
\end{theorem}
\begin{proof}
Take any $u, v \in V$ and let $w = u - v$.
We define $w_0 \in V_0$ and $w_k \in V_k$, $1 \leq k \leq N$, as
\begin{equation*}
    w_0 = J_H w, \quad
    w_k = I_h \left( \theta_k (w - R_0^* w_0) \right),
    \quad 1 \leq k \leq N,
\end{equation*}
where $\{ \theta_k \}_{k=1}^N$ is the partition of unity given in~\eqref{pou}.
As we use the space decomposition~\eqref{space_decomp_2L} consisting of $N+1$ subspaces, we have to verify the following:
\begin{subequations}
\begin{align}
\label{stable_2L_1}
w &= \sum_{k=0}^N R_k^* w_k, \\
\label{stable_2L_2}
\sum_{k=0}^N D_{F_h} (v + R_k^* w_k, v) &\leq \frac{C_0^2}{2} | w |_{H^1 (\Omega)}^2, \\
\label{stable_2L_3}
\sum_{k=0}^N G_h (v + R_k^* w_k) &\leq G_h (u) + N G_h (v).
\end{align}
\end{subequations}
Among the above equations,~\eqref{stable_2L_1} is easy to verify.
Now, we show that~\eqref{stable_2L_3} holds through a similar argument as in~\cite[Proposition~5.2]{Badea:2010}.
Since $J_H$ is positivity-preserving~(see~\eqref{positivity}), there exists a function $\theta_w \in C^0 (\overline{\Omega})$, such that $0 \leq \theta_w \leq 1$ and $J_H w = \theta_w w$~\cite[equation~(61)]{BK:2012}.
That is, we have
\begin{equation*}
    w_0 (x) = \theta_w (x) w(x), \quad
    w_k (x) = (1 - \theta_w (x)) \theta_k (x) w(x), \quad
    1 \leq k \leq N,
\end{equation*}
for each $x \in \cN_h$.
It follows by~\eqref{G_pointwise} and the convexity of $\phi_h^x$ that
\begin{equation} \begin{split}
    \label{Thm1:2L}
    G_h ( v + R_0^* w_0) &= \sum_{x \in \cN_h} \phi_h^x \left( v(x) + \theta_w (x) w(x) \right) \\
    &\leq \sum_{x \in \cN_h} \left[ \theta_w (x) \phi_h^x ( v(x) + w(x)) + (1 - \theta_w (x)) \phi_h^x (v(x)) \right]
\end{split} \end{equation}
and that
\begin{equation} \begin{split}
    \label{Thm2:2L}
    &\sum_{k=1}^N G_h (v + R_k^* w_k)
    = \sum_{x \in \cN_h} \sum_{k=1}^N \phi_h^x \left( v(x) + (1 - \theta_w (x))\theta_k (x) w(x) \right) \\
    &\leq \sum_{x \in \cN_h} \sum_{k=1}^N \left[ (1 - \theta_w (x)) \theta_k (x) \phi_h^x (v(x) + w(x) ) + \left( 1 - \theta_k (x) + \theta_w (x) \theta_k (x) \right) \phi_h^x (v(x)) \right] \\
    &= \sum_{x \in \cN_h} \left[ (1 - \theta_w (x)) \phi_h^x (v(x) + w(x)) + ( N-1 + \theta_w (x) ) \phi_h^x (v(x)) \right].
\end{split} \end{equation}
Summing~\eqref{Thm1:2L} and~\eqref{Thm2:2L} yields~\eqref{stable_2L_3} as follows:
\begin{equation*}
    \sum_{k=0}^N G_h (v+ R_k^* w_k)
    \leq \sum_{x \in \cN_h} \left[ \phi_h^x (v(x) + w(x)) + N \phi_h^x (v(x)) \right]
    = G_h (u) + N G_h (v).
\end{equation*}

Next, we prove~\eqref{stable_2L_2} by proceeding similarly to~\cite[Theorem~4.9]{Park:2023}.
By~\eqref{F_distance} and \cref{Lem:J_H}, we have
\begin{equation}
    \label{Thm3:2L}
    D_{F_h} (v + R_0^* w_0 , v)
    \lesssim  \intO | \nabla (R_0^* w_0 ) |^2 \,dx
    =  |J_H w|_{H^1 (\Omega)}^2
    \lesssim C_d (H,h)^2 |w|_{H^1 (\Omega)}^2.
\end{equation}
Using the same argument as in the proof of \cref{Thm:1L}, for each $1 \leq k \leq N$, we obtain
\begin{equation}
    \label{Thm4:2L}
    D_{F_h} (v + R_k^* w_k, v) \lesssim \left( \left( 1 + \frac{H}{\delta} \right) |w - J_H w|_{H^1 (\Omega)}^2 + \frac{1}{H \delta} \| w - J_H w \|_{L^2 (\Omega)}^2 \right).
\end{equation}
Meanwhile, \cref{Lem:J_H} implies that
\begin{equation}
    \label{Thm5:2L}
    | w - J_H w|_{H^1 (\Omega)}^2
    \lesssim |w|_{H^1 (\Omega)}^2 + |J_H w|_{H^1 (\Omega)}^2
    \lesssim C_d (H, h)^2 |w|_{H^1 (\Omega)}^2
\end{equation}
and that
\begin{equation}
    \label{Thm6:2L}
    \frac{1}{H \delta} \| w - J_H w \|_{L^2 (\Omega)}^2
    \lesssim C_d (H, h)^2 \frac{H}{\delta} |w |_{H^1 (\Omega)}^2.
\end{equation}
Combining~\eqref{Thm3:2L}--\eqref{Thm6:2L} yields~\eqref{stable_2L_2}.
\end{proof}

As in the one-level case, the constant $C_0$ for the two-level case analyzed in \cref{Thm:2L} is independent of the nonlinear term in $E_h$.
Moreover, \cref{Thm:2L} implies that the two-level method is scalable in the sense that the convergence rate depends on $H/h$ and $H/ \delta$.
Namely, in both cases of small overlap $\delta \approx h$ and generous overlap $\delta \approx H$, the convergence rate is uniformly bounded by a function of $H/h$.

We conclude this section by summarizing the linear convergence rates of the one- and two-level additive Schwarz methods for~\eqref{model_FEM} derived from \cref{Thm:1L,Thm:2L} in the following.

\begin{corollary}
    \label{Cor:conv}
    Suppose that \cref{Ass:stable,Ass:convex} hold.
    In \cref{Alg:ASM}, if $\tau \in (0, \tau_0]$, then we have
    \begin{equation}
    \label{Cor1:conv}
        \frac{E_h (u^{(n+1)}) - E_h (u_h)}{E_h (u^{(n)}) - E_h (u_h)} \leq 1 - \tau \min \left\{ 1, \frac{1}{ C_0^2} \right\},
        \quad n \geq 0,
    \end{equation}
    where
    \begin{equation*}
        C_0^2 \approx
        \begin{cases}
            \dfrac{1}{H \delta}, & \text{ for one-level}, \\
            C_d (H, h)^2 \left( 1 + \dfrac{H}{\delta} \right), & \text{ for two-level},
        \end{cases}
    \end{equation*}
    and $C_d (H,h)$ was given in~\eqref{C_d}.
\end{corollary}

\begin{remark}
\label{Rem:nonsmooth}
It is worth noting that the only properties of $G_h$ that are used in the analysis presented in this section are the separable property~\eqref{G_pointwise} and the convexity of $\phi_h^x$.
As a result, the findings in this section are applicable even when $\phi$ in~\eqref{model} is nonsmooth.
Examples of such problems include variational inequalities~\cite{BK:2012,LP:2021} and $L^1$-penalized variational problems~\cite{TSFO:2015}.
\end{remark}

\section{Numerical results}
\label{Sec:Numerical}
In the section, we present numerical results that verify our theoretical findings. 
All the algorithms were programmed using MATLAB~R2022b and performed on a desktop equipped with AMD Ryzen~5 5600X CPU~(3.7GHz, 6C), 40GB RAM, and the operating system Windows~10 Pro.

\subsection{Monomial nonlinearity}
As the first example, we consider the following semilinear boundary value problem, which was introduced in \cref{Ex:monomial}:
\begin{equation} \begin{split}
\label{monomial}
- \Delta u + \alpha |u|^{m-2} u = g \quad &\text{ in } \Omega, \\
u = 0 \quad &\text{ on } \partial \Omega,
\end{split} \end{equation}
where $\alpha \geq 0$ and $m \in \mathbb{Z}_{\geq 2}$.
In~\eqref{monomial}, we set $\Omega = (0,1)^2 \subset \mathbb{R}^2$, and $g$ is chosen so that the exact solution $u$ is given by $u(x,y) = x (1-x) \sin \pi y$.
As in~\cite[Appendix~A]{LP:2022}, we partition the domain $\Omega$ into $2 \times 1/H \times 1/H$ uniform triangles to create a coarse triangulation $\cT_H$, and then refine $\cT_H$ to obtain a fine triangulation $\cT_h$, consisting of a total of $2 \times 1/h \times 1/h$ uniform triangles.
Each nonoverlapping subdomain $\Omega_k$, $1 \leq k \leq N = 1/H \times 1/H$, is defined by a rectangular region composed of two coarse triangles that share a diagonal edge.
We then obtain the corresponding overlapping subdomain $\Omega_l'$ by extending $\Omega_k$ with surrounding layers of fine triangles from $\cT_h$ with width $\delta$.
Note that the domain decomposition $\{ \Omega_k' \}_{k=1}^N$ can be colored with four colors if $\delta$ is small enough.
Consequently, we set the step size $\tau$ of \cref{Alg:ASM} as follows:
\begin{equation*}
\tau = \begin{cases}
\frac{1}{4}, \quad & \text{in the one-level method,} \\
\frac{1}{5}, \quad & \text{in the two-level method.}
\end{cases}
\end{equation*}

In \cref{Alg:ASM}, we initialize $u^{(0)}$ as zero.
For solving the local and coarse problems defined on $V_k$, $0 \leq k \leq N$, we utilize the damped Newton method~\cite{BV:2004} along with the stop criterion given by
\begin{equation}
\label{stop}
\left| \frac{E_k^{(n)} (w_k^{(j+1)}) - E_k^{(n)} (w_k^{(j)})}{E_k^{(n)} (w_k^{(j+1)})} \right| < 10^{-12}.
\end{equation}
Here $E_k^{(n)}$ is the energy functional associated with the subspace $V_k$ at the $n$th outer iteration, and $j$ is the number of inner iterations.
To establish a reference solution $u_h$, we perform a sufficient number of iterations of the damped Newton method applied to the full-dimension problem~\eqref{model_FEM}.

\begin{table}
\centering
\caption{Geometric averages of the linear convergence rates $\frac{E_h (u^{(n+1)}) - E_h (u_h)}{E_h (u^{(n)}) - E_h (u_h)}$ over 30 iterations of additive Schwarz methods for the monomial nonlinearity problem~\eqref{monomial}.}
\resizebox{\textwidth}{!}{
\begin{tabular}{c|cccccc|cccc}
 \cline{1-5} \cline{7-11}
 $m$ & $\alpha = 10^0$ & $\alpha = 10^1$ & $\alpha = 10^2$ & $\alpha = 10^3$ & & 
 $m$ & $\alpha = 10^0$ & $\alpha = 10^1$ & $\alpha = 10^2$ & $\alpha = 10^3$ \\
 \cline{1-5} \cline{7-11}
 3 & 0.9183 & 0.9109 & 0.8391 & 0.6226 & &
 3 & 0.7134 & 0.7036 & 0.6534 & 0.5742 \\
 6 & 0.9191 & 0.9190 & 0.9184 & 0.9114 & &
 6 & 0.7146 & 0.7144 & 0.7126 & 0.6988 \\
 9 & 0.9191 & 0.9191 & 0.9191 & 0.9190 & &
 9 & 0.7146 & 0.7146 & 0.7145 & 0.7143 \\ 
 12 & 0.9191 & 0.9191 & 0.9191 & 0.9191 & & 
 12 & 0.7146 & 0.7146 & 0.7146 & 0.7146 \\
 \cline{1-5} \cline{7-11}
 \multicolumn{5}{c}{(a)~One-level, $h = 2^{-5}$, $H = 2^{-2}$, $\delta = 2h$ } & &
 \multicolumn{5}{c}{(b)~Two-level, $h = 2^{-5}$, $H = 2^{-2}$, $\delta = 2h$} \vspace{0.3cm}\\

 \cline{1-5} \cline{7-11}
 $m$ & $\alpha = 10^0$ & $\alpha = 10^1$ & $\alpha = 10^2$ & $\alpha = 10^3$ & & 
 $m$ & $\alpha = 10^0$ & $\alpha = 10^1$ & $\alpha = 10^2$ & $\alpha = 10^3$ \\
 \cline{1-5} \cline{7-11}
 3 & 0.9773 & 0.9757 & 0.9568 & 0.7950 & &
 3 & 0.6753 & 0.6712 & 0.6477 & 0.5917 \\
 6 & 0.9775 & 0.9774 & 0.9774 & 0.9766 & &
 6 & 0.6757 & 0.6757 & 0.6751 & 0.6708 \\
 9 & 0.9775 & 0.9775 & 0.9775 & 0.9774 & &
 9 & 0.6757 & 0.6757 & 0.6757 & 0.6757 \\ 
 12 & 0.9775 & 0.9775 & 0.9775 & 0.9775 & & 
 12 & 0.6757 & 0.6757 & 0.6757 & 0.6757 \\
 \cline{1-5} \cline{7-11}
 \multicolumn{5}{c}{(c)~One-level, $h = 2^{-6}$, $H = 2^{-3}$, $\delta = 2h$ } & &
 \multicolumn{5}{c}{(d)~Two-level, $h = 2^{-6}$, $H = 2^{-3}$, $\delta = 2h$} \\
\end{tabular}
}
\label{Table:monomial}
\end{table}

\cref{Table:monomial} presents a summary of the geometric averages of the linear convergence rates.
These rates are evaluated as the ratios $\frac{E_h (u^{(n+1)}) - E_h (u_h)}{E_h (u^{(n)}) - E_h (u_h)}$,  computed over 30~iterations of one- and two-level additive Schwarz methods for~\eqref{monomial}.
Across all cases, we consistently observe that the convergence rates exhibit a uniform upper bound regardless of the values of $\alpha$ and $m$.
More precisely, while the convergence rates generally increase when with larger values of $m$, they eventually stabilize as $m$ becomes sufficiently large.
This numerical observation verifies the central claim of this paper, shown in \cref{Cor:conv}, namely that the convergence rates of additive Schwarz methods are independent of the nonlinearity of the problem.

\begin{figure}
\resizebox{\textwidth}{!}{
  \subfloat[][One-level , $H/h = 2^3$, $\delta = 2h$]{\includegraphics[width=0.34\hsize]{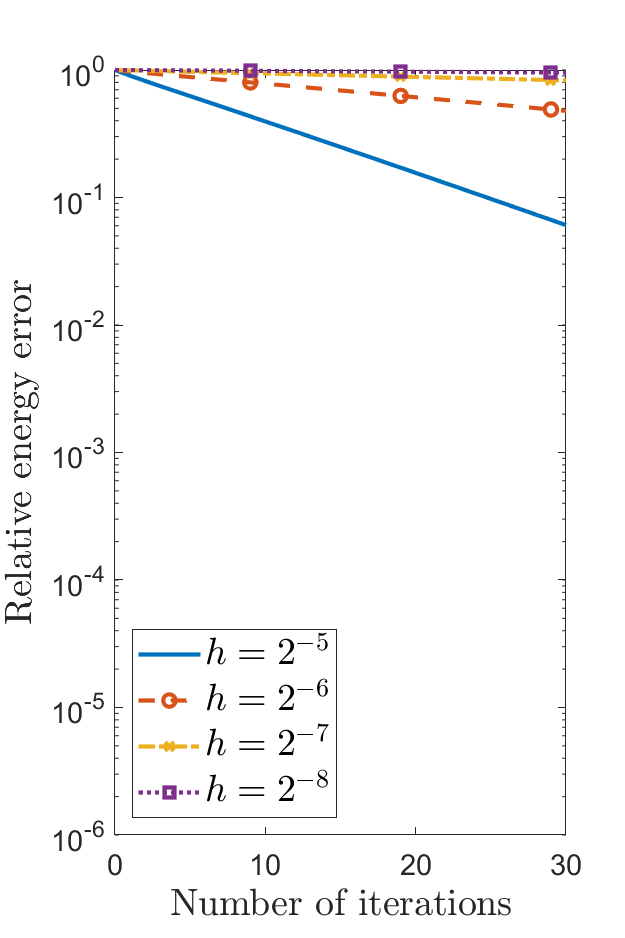}}
  \subfloat[][Two-level , $H/h = 2^3$, $\delta = 2h$]{\includegraphics[width=0.34\hsize]{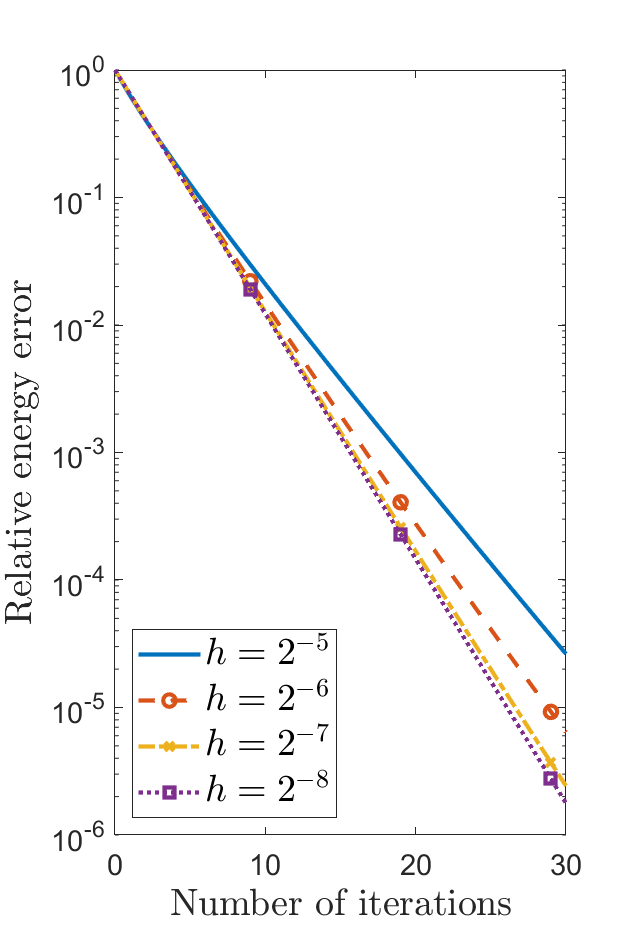}}
  \subfloat[][One-level , $H/h = 2^3$, $\delta = 4h$]{\includegraphics[width=0.34\hsize]{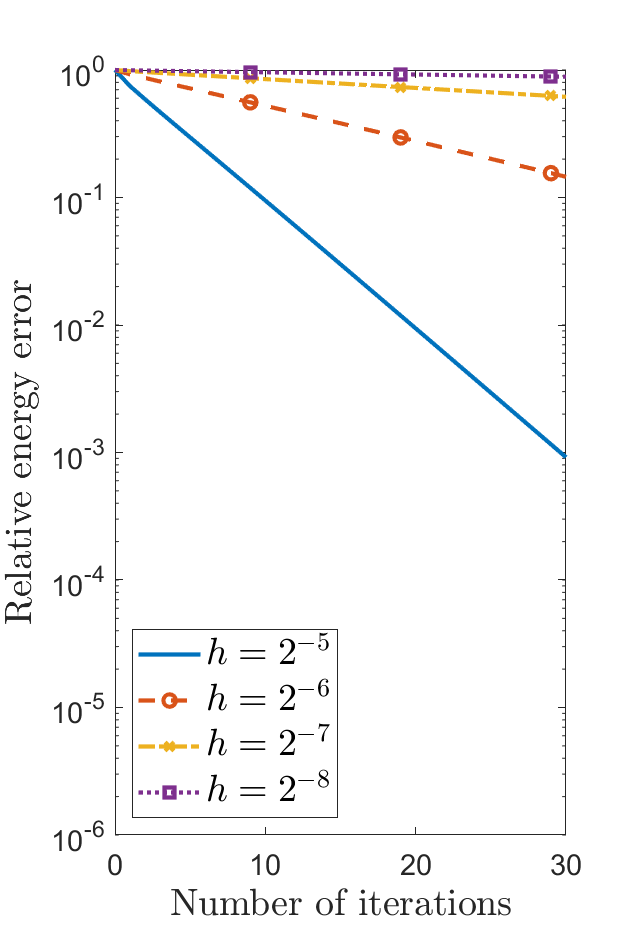}}
  \subfloat[][Two-level , $H/h = 2^3$, $\delta = 4h$]{\includegraphics[width=0.34\hsize]{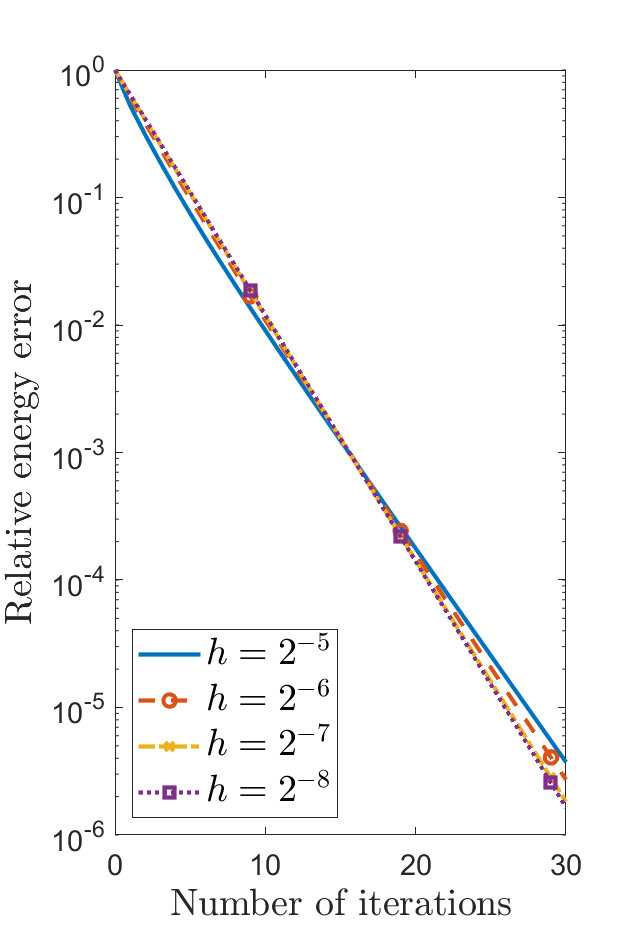}}
}
  \\
\resizebox{\textwidth}{!}{
  \subfloat[][One-level , $H/h = 2^4$, $\delta = 2h$]{\includegraphics[width=0.34\hsize]{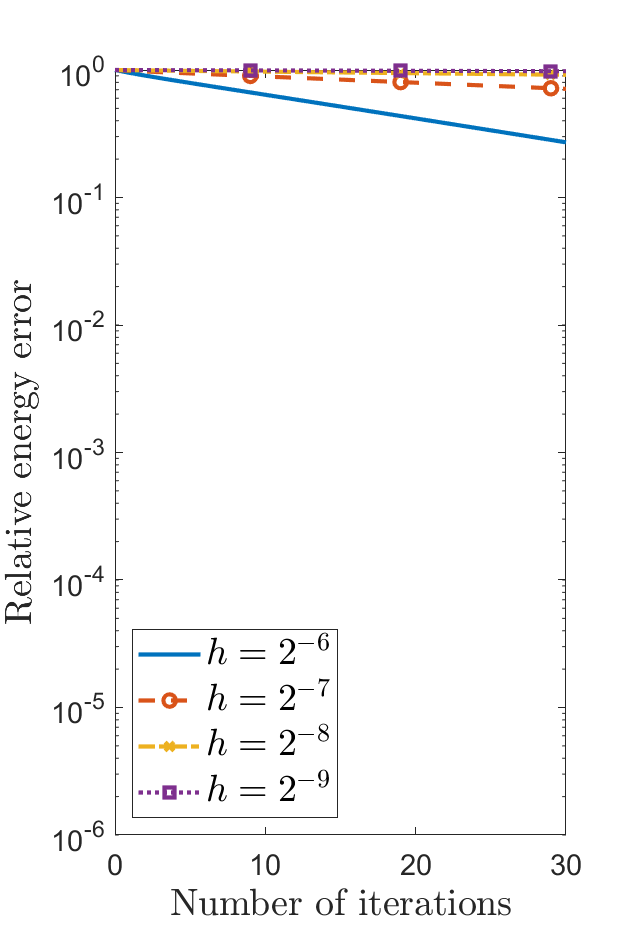}}
  \subfloat[][Two-level , $H/h = 2^4$, $\delta = 2h$]{\includegraphics[width=0.34\hsize]{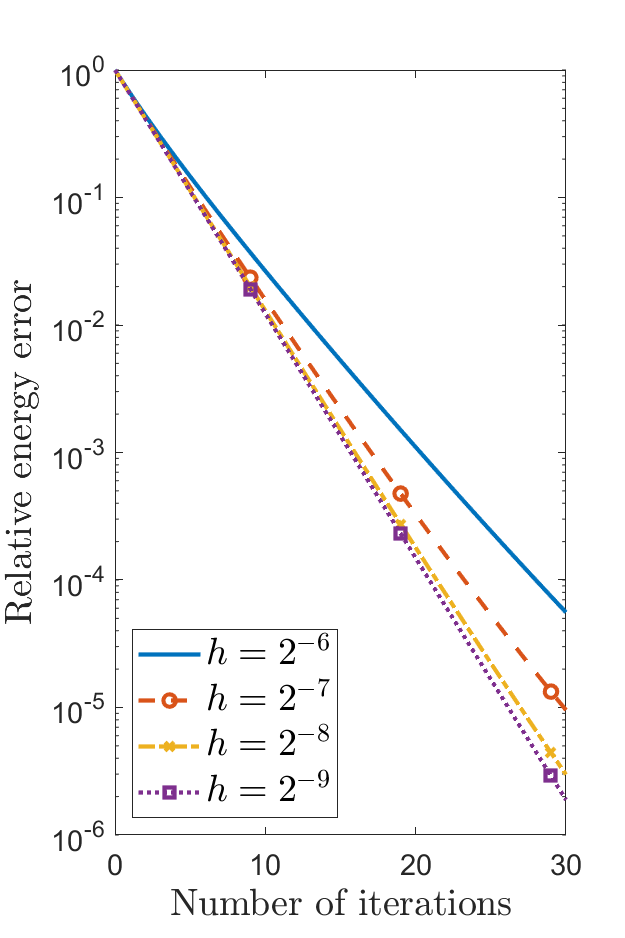}}
  \subfloat[][One-level , $H/h = 2^4$, $\delta = 2h$]{\includegraphics[width=0.34\hsize]{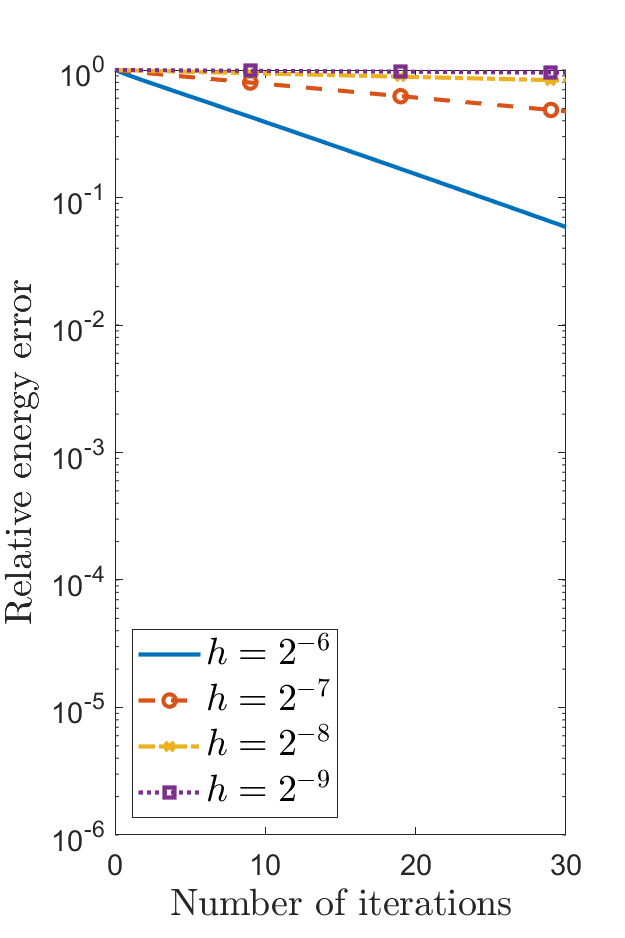}}
  \subfloat[][Two-level , $H/h = 2^4$, $\delta = 4h$]{\includegraphics[width=0.34\hsize]{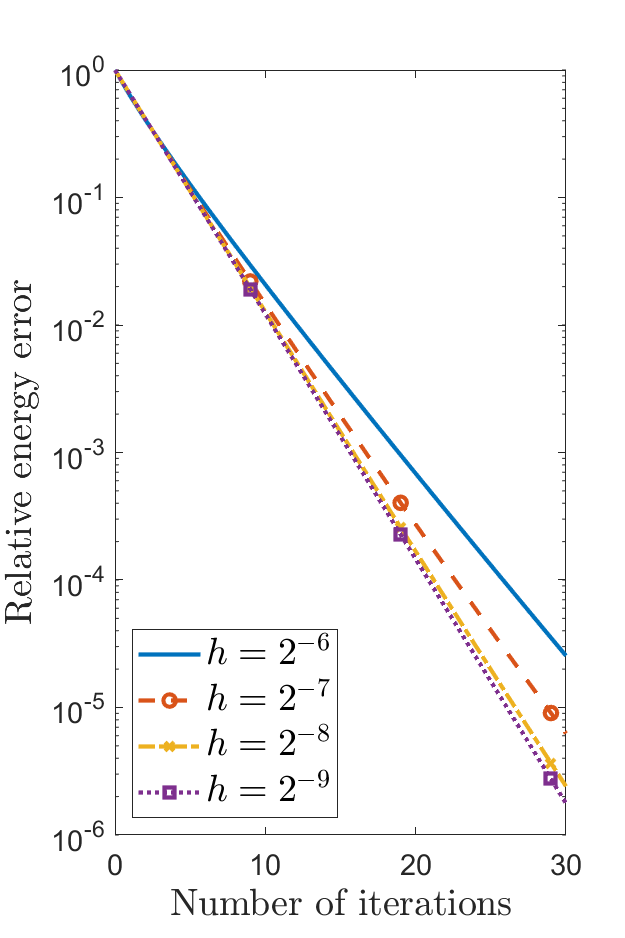}}
}
  \caption{Decay of the relative energy error $\frac{E_h (u^{(n)}) - E_h (u_h)}{|E_h (u_h)|}$ in additive Schwarz methods for the monomial nonlinearity problem~\eqref{monomial}~($m = 3$, $\alpha = 10^1$).}
  \label{Fig:monomial}
\end{figure}

In \cref{Fig:monomial}, the energy error $E_h (u^{(n)}) - E_h (u_h)$ is plotted for various combinations of $h$, $H$, and $\delta$.
It is observed that the convergence rate of the one-level method significantly deteriorates as $h$ decreases.
This observation aligns with \cref{Thm:1L}, which states that the convergence rate of the one-level method becomes dependent on $h^{-2}$ under fixed values of $H/h$ and $H/\delta$.
In contrast, the convergence rate of the two-level method remains stable even as $h$ becomes small.
In particular, the convergence curves for $h = 2^{-8}$ and $h = 2^{-9}$ almost overlap in every case.
These numerical results verify \cref{Thm:2L}, which states that the convergence rate of the two-level method is uniformly bounded when both $H/h$ and $H/\delta$ are fixed.

\subsection{Nonlinear Poisson--Boltzmann equation}
Next, we consider the following nonlinear Poisson--Boltzmann equation, which was introduced in \cref{Ex:PB}:
\begin{equation} \begin{split}
\label{PB}
- \Delta u + \sinh \alpha u = g \quad &\text{ in } \Omega, \\
u = 0 \quad &\text{ on } \partial \Omega,
\end{split} \end{equation}
where $\alpha > 0$.
In~\eqref{PB}, we set $\Omega = (0, 1)^2 \subset \mathbb{R}^2$ and determine $g$ such that it yields the exact solution $u(x,y) = x (1-x) \sin \pi y$.
For discretization, domain decomposition, and the details of \cref{Alg:ASM}, we adopt the same settings as used in~\eqref{monomial}.

\begin{table}
\centering
\caption{Geometric averages of the linear convergence rates $\frac{E_h (u^{(n+1)}) - E_h (u_h)}{E_h (u^{(n)}) - E_h (u_h)}$ over 30 iterations of additive Schwarz methods for the nonlinear Poisson--Boltzmann equation~\eqref{PB}.}
\resizebox{0.9\textwidth}{!}{
\begin{tabular}{ccccccccc}
 \cline{1-4} \cline{6-9}
 $\alpha = 10^{-2}$ & $\alpha = 10^{-1}$ & $\alpha = 10^0$ & $\alpha = 10^1$ & & 
 $\alpha = 10^{-2}$ & $\alpha = 10^{-1}$ & $\alpha = 10^0$ & $\alpha = 10^1$ \\
 \cline{1-4} \cline{6-9}
 0.8112 & 0.8106 & 0.8046 & 0.6831 & &
 0.6705 & 0.6702 & 0.6665 & 0.6167 \\
 \cline{1-4} \cline{6-9}
 \multicolumn{4}{c}{(a)~One-level, $h = 2^{-5}$, $H = 2^{-2}$, $\delta = 2h$ } & &
 \multicolumn{4}{c}{(b)~Two-level, $h = 2^{-5}$, $H = 2^{-2}$, $\delta = 2h$} \vspace{0.3cm}\\

 \cline{1-4} \cline{6-9}
 $\alpha = 10^{-2}$ & $\alpha = 10^{-1}$ & $\alpha = 10^0$ & $\alpha = 10^1$ & & 
 $\alpha = 10^{-2}$ & $\alpha = 10^{-1}$ & $\alpha = 10^0$ & $\alpha = 10^1$ \\
 \cline{1-4} \cline{6-9}
 0.9433 & 0.9430 & 0.9407 & 0.8905 & &
 0.6574 & 0.6572 & 0.6554 & 0.6277 \\
 \cline{1-4} \cline{6-9}
 \multicolumn{4}{c}{(c)~One-level, $h = 2^{-6}$, $H = 2^{-3}$, $\delta = 2h$ } & &
 \multicolumn{4}{c}{(d)~Two-level, $h = 2^{-6}$, $H = 2^{-3}$, $\delta = 2h$} \\
\end{tabular}
}
\label{Table:PB}
\end{table}

\begin{figure}
\resizebox{\textwidth}{!}{
  \subfloat[][One-level , $H/h = 2^3$, $\delta = 2h$]{\includegraphics[width=0.34\hsize]{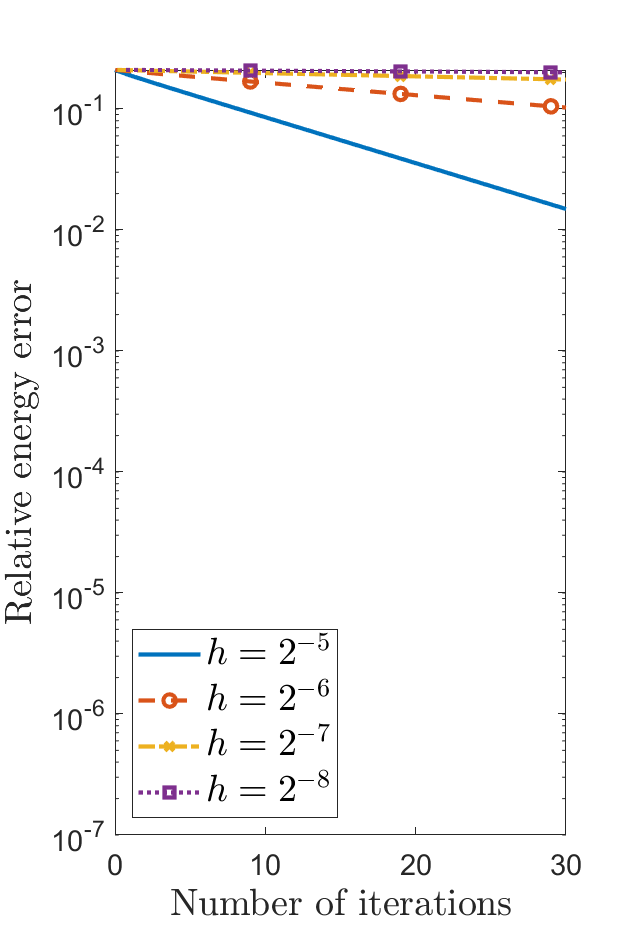}}
  \subfloat[][Two-level , $H/h = 2^3$, $\delta = 2h$]{\includegraphics[width=0.34\hsize]{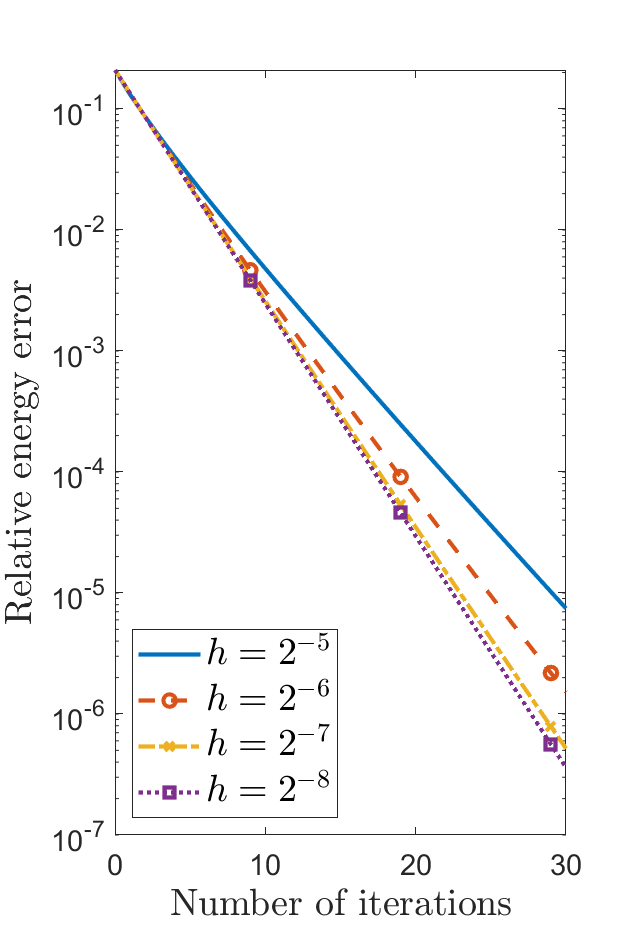}}
  \subfloat[][One-level , $H/h = 2^3$, $\delta = 4h$]{\includegraphics[width=0.34\hsize]{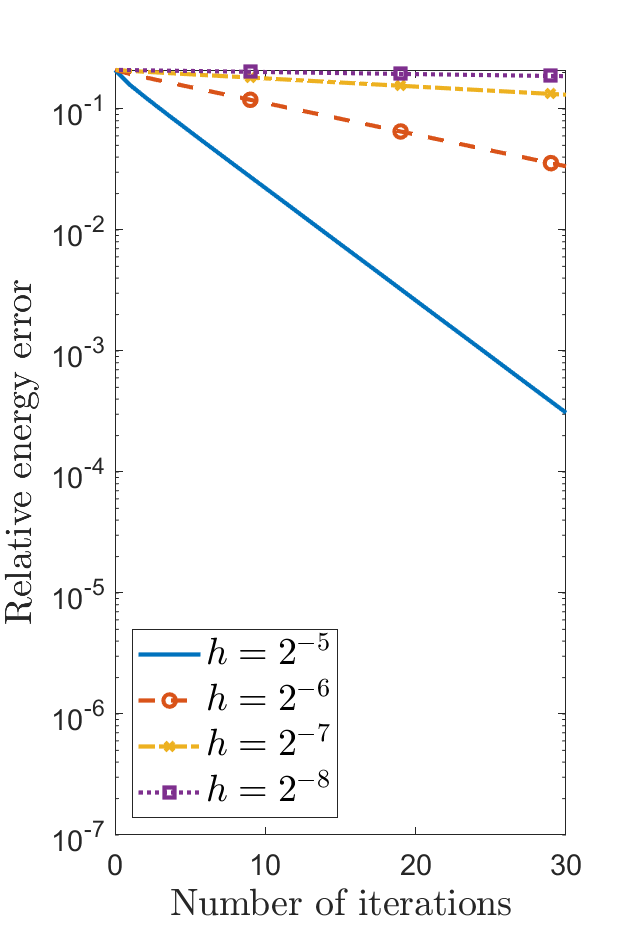}}
  \subfloat[][Two-level , $H/h = 2^3$, $\delta = 4h$]{\includegraphics[width=0.34\hsize]{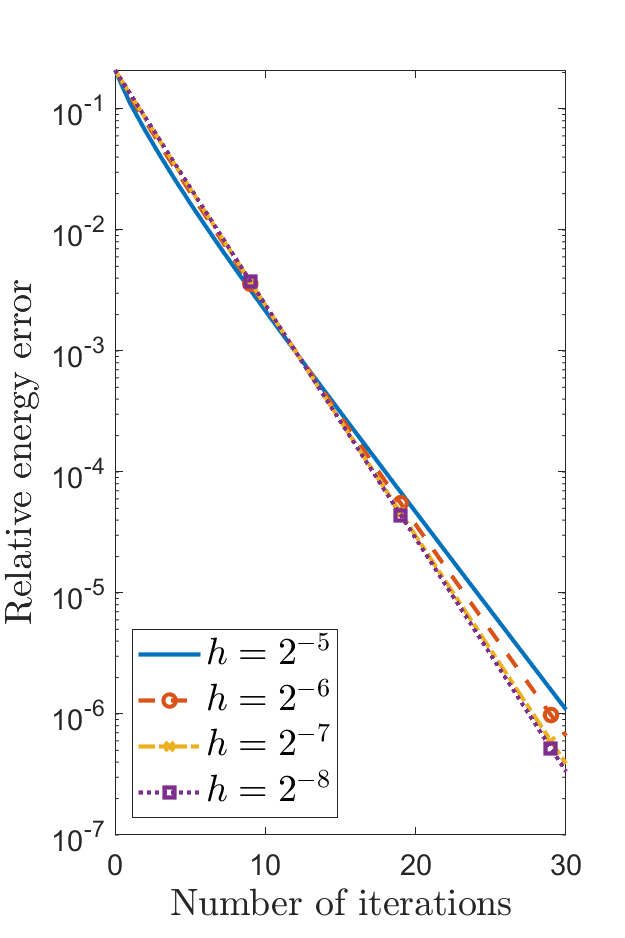}}
}
  \\
\resizebox{\textwidth}{!}{
  \subfloat[][One-level , $H/h = 2^4$, $\delta = 2h$]{\includegraphics[width=0.34\hsize]{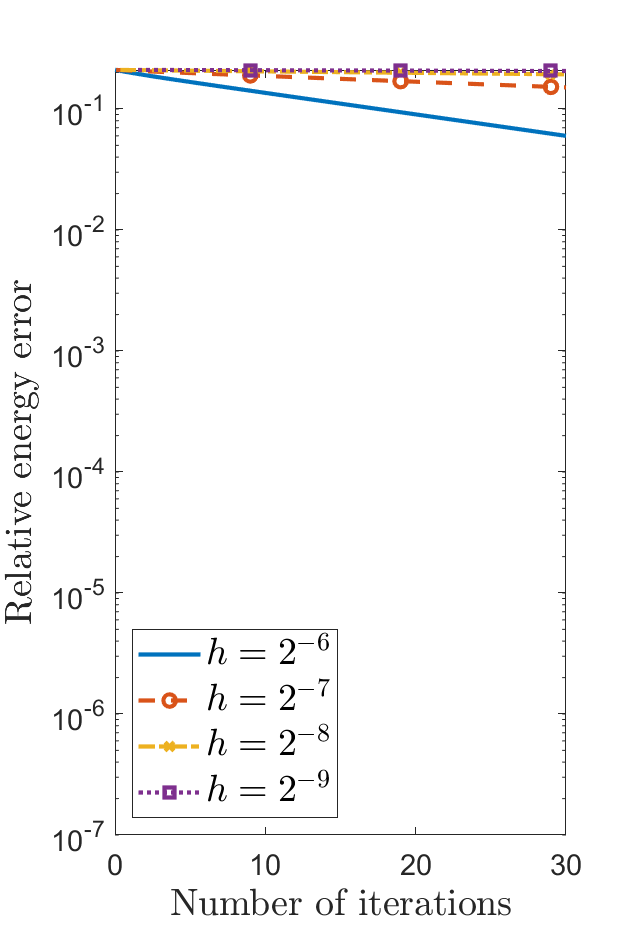}}
  \subfloat[][Two-level , $H/h = 2^4$, $\delta = 2h$]{\includegraphics[width=0.34\hsize]{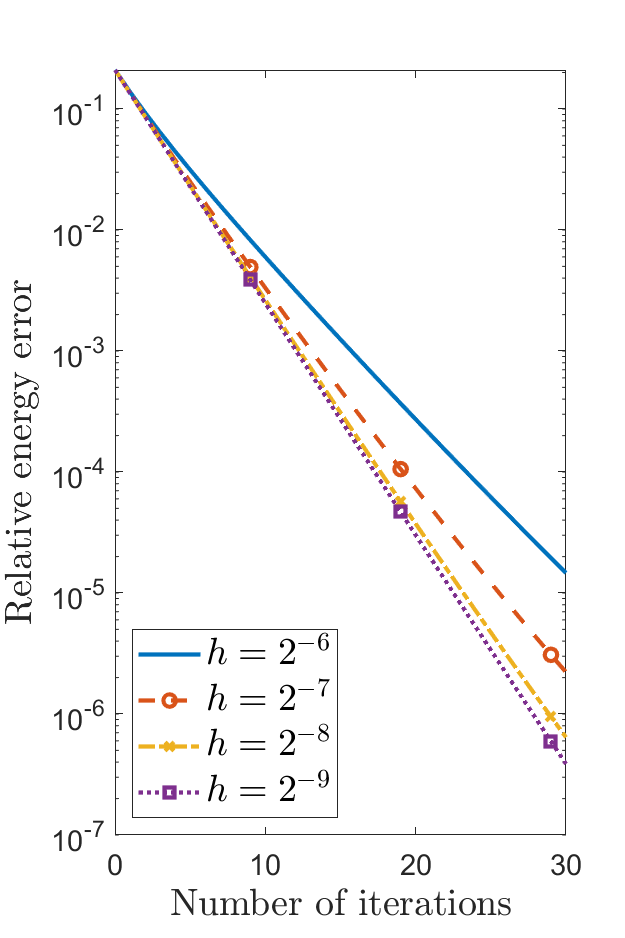}}
  \subfloat[][One-level , $H/h = 2^4$, $\delta = 2h$]{\includegraphics[width=0.34\hsize]{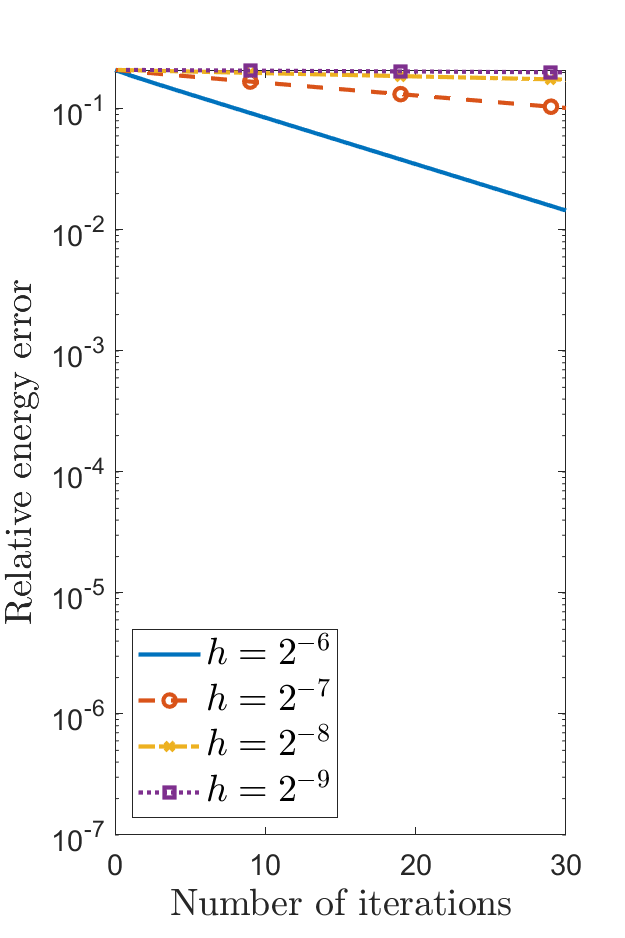}}
  \subfloat[][Two-level , $H/h = 2^4$, $\delta = 4h$]{\includegraphics[width=0.34\hsize]{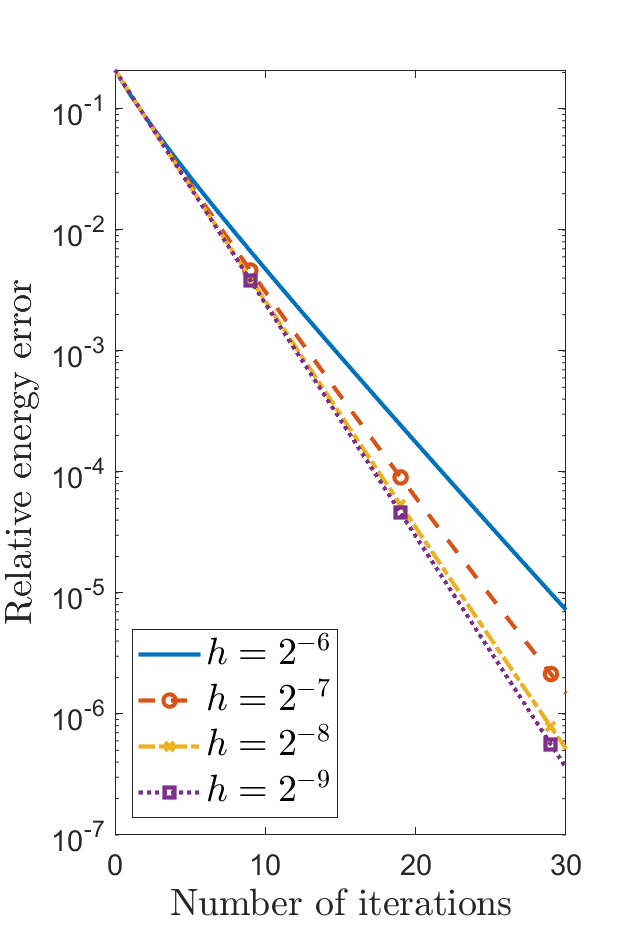}}
}
  \caption{Decay of the relative energy error $\frac{E_h (u^{(n)}) - E_h (u_h)}{|E_h (u_h)|}$ in additive Schwarz methods for the nonlinear Poisson--Boltzmann equation~\eqref{PB}~($\alpha = 10^0$).}
  \label{Fig:PB}
\end{figure}

The convergence results of the one- and two-level additive Schwarz methods for~\eqref{PB} are presented in \cref{Table:PB} and \cref{Fig:PB}.
In \cref{Table:PB}, we observe a decreasing trend in convergence rates as $\alpha$ takes larger values.
This phenomenon is attributed to the increasing dominance of the nonlinear term, which does not impact the convergence rate, as $\alpha$ grows.
In particular, we find that the convergence rates are bounded above regardless of the values of $\alpha$.
In \cref{Fig:PB}, the convergence rate of the two-level method remains stable as $h$ diminishes, whereas that of the one-level method deteriorates, thereby agreeing with \cref{Thm:1L,Thm:2L}.

\begin{table}
\centering
\caption{Maximum numbers of Newton iterations to solve the local and coarse problems during 30 iterations of additive Schwarz methods for the nonlinear Poisson--Boltzmann equation~\eqref{PB}.}
\resizebox{0.9\textwidth}{!}{
\begin{tabular}{ccccccccc}
 \cline{1-4} \cline{6-9}
 $\alpha = 10^{-2}$ & $\alpha = 10^{-1}$ & $\alpha = 10^0$ & $\alpha = 10^1$ & & 
 $\alpha = 10^{-2}$ & $\alpha = 10^{-1}$ & $\alpha = 10^0$ & $\alpha = 10^1$ \\
 \cline{1-4} \cline{6-9}
 2 & 2 & 2 & 4 & &
 2 & 2 & 3 & 5 \\
 \cline{1-4} \cline{6-9}
 \multicolumn{4}{c}{(a)~One-level, $h = 2^{-5}$, $H = 2^{-2}$, $\delta = 2h$ } & &
 \multicolumn{4}{c}{(b)~Two-level, $h = 2^{-5}$, $H = 2^{-2}$, $\delta = 2h$} \vspace{0.3cm}\\

 \cline{1-4} \cline{6-9}
 $\alpha = 10^{-2}$ & $\alpha = 10^{-1}$ & $\alpha = 10^0$ & $\alpha = 10^1$ & & 
 $\alpha = 10^{-2}$ & $\alpha = 10^{-1}$ & $\alpha = 10^0$ & $\alpha = 10^1$ \\
 \cline{1-4} \cline{6-9}
 2 & 2 & 2 & 3 & &
 2 & 2 & 3 & 6 \\
 \cline{1-4} \cline{6-9}
 \multicolumn{4}{c}{(c)~One-level, $h = 2^{-6}$, $H = 2^{-3}$, $\delta = 2h$ } & &
 \multicolumn{4}{c}{(d)~Two-level, $h = 2^{-6}$, $H = 2^{-3}$, $\delta = 2h$} \\
\end{tabular}
}
\label{Table:Newton}
\end{table}

It is worth noting that the nonlinearity in the energy functional affects on the local and coarse problems, whereas it does not affect on the convergence of additive Schwarz methods.
In \cref{Table:Newton}, we summarize the maximum numbers of Newton iterations required to solve the local and coarse problems during 30 iterations of additive Schwarz methods for~\eqref{PB}.
We observe that the maximum number of Newton iterations increases with higher values of $\alpha$, indicating that stronger nonlinearity corresponds to more Newton iterations.
This suggests that the nonlinearity is fully localized, affecting only the local problems and leaving the convergence of additive Schwarz methods unaffected.
Meanwhile, the maximum numbers of Newton iterations do not differ significantly with respect to the mesh size $h$.
This phenomenon aligns with the well-known mesh-independence principle, as discussed in, e.g.,~\cite{ABPR:1986}.

\subsection{\texorpdfstring{$L^1$}{L1}-penalized problem}
As mentioned in \cref{Rem:nonsmooth}, the convergence theorems presented in \cref{Sec:DD} remain valid even when $\phi$ in~\eqref{model} is nonsmooth.
To numerically verify this, we consider the following $L^1$-penalized variational problem~\cite{TSFO:2015}:
\begin{equation}
\label{L1}
    \min_{u \in H_0^1 (\Omega)} \intO \left( \frac{1}{2} | \nabla u |^2 - gu + \alpha | u | \right) \,dx,
\end{equation}
where $\alpha > 0$.
In~\eqref{L1}, we set $\Omega = (0,1)^2 \subset \mathbb{R}^2$ and $g(x,y) = 10^3 x(1-x) \sin \pi y$.
For discretization and domain decomposition, we adopt the same settings as used in~\eqref{monomial}.

As the problem~\eqref{L1} is nonsmooth, the Newton method is not applicable.
Instead, we employ the fast gradient method with gradient adaptive restart~(AFGM)~\cite{OC:2015}, which is known to be an efficient first-order optimization algorithm for composite convex optimization; see also~\cite{KF:2018,Park:2021}.
For solving the local problems, we utilize AFGM along with the stop criterion~\eqref{stop}.
To tackle the coarse problems, we apply AFGM to the dual formulations of the coarse problems, akin to the approach introduced in~\cite[Proposition~5.1]{Park:2023}.
A reference solution $u_h$ is obtained through a sufficient number of iterations of the AFGM.

\begin{table}
\centering
\caption{Geometric averages of the linear convergence rates $\frac{E_h (u^{(n+1)}) - E_h (u_h)}{E_h (u^{(n)}) - E_h (u_h)}$ over 30 iterations of additive Schwarz methods for the $L^1$-penalized problem~\eqref{L1}.}
\resizebox{0.9\textwidth}{!}{
\begin{tabular}{m{1.2cm}m{1.2cm}m{1.2cm}m{1.2cm}cm{1.2cm}m{1.2cm}m{1.2cm}m{1.2cm}}
 \cline{1-4} \cline{6-9}
 $\alpha = 10$ & $\alpha = 20$ & $\alpha = 30$ & $\alpha = 40$ & & 
 $\alpha = 10$ & $\alpha = 20$ & $\alpha = 30$ & $\alpha = 40$ \\
 \cline{1-4} \cline{6-9}
 0.9193 & 0.9192 & 0.9192 & 0.9191 & &
 0.7148 & 0.7148 & 0.7148 & 0.7148 \\
 \cline{1-4} \cline{6-9}
 \multicolumn{4}{c}{(a)~One-level, $h = 2^{-5}$, $H = 2^{-2}$, $\delta = 2h$ } & &
 \multicolumn{4}{c}{(b)~Two-level, $h = 2^{-5}$, $H = 2^{-2}$, $\delta = 2h$} \vspace{0.3cm}\\

 \cline{1-4} \cline{6-9}
 $\alpha = 10$ & $\alpha = 20$ & $\alpha = 30$ & $\alpha = 40$ & & 
 $\alpha = 10$ & $\alpha = 20$ & $\alpha = 30$ & $\alpha = 40$ \\
 \cline{1-4} \cline{6-9}
 0.9777 & 0.9777 & 0.9776 & 0.9775 & &
 0.6755 & 0.6755 & 0.6755 & 0.6755 \\
 \cline{1-4} \cline{6-9}
 \multicolumn{4}{c}{(c)~One-level, $h = 2^{-6}$, $H = 2^{-3}$, $\delta = 2h$ } & &
 \multicolumn{4}{c}{(d)~Two-level, $h = 2^{-6}$, $H = 2^{-3}$, $\delta = 2h$} \\
\end{tabular}
}
\label{Table:L1}
\end{table}

\begin{figure}
\resizebox{\textwidth}{!}{
  \subfloat[][One-level , $H/h = 2^3$, $\delta = 2h$]{\includegraphics[width=0.34\hsize]{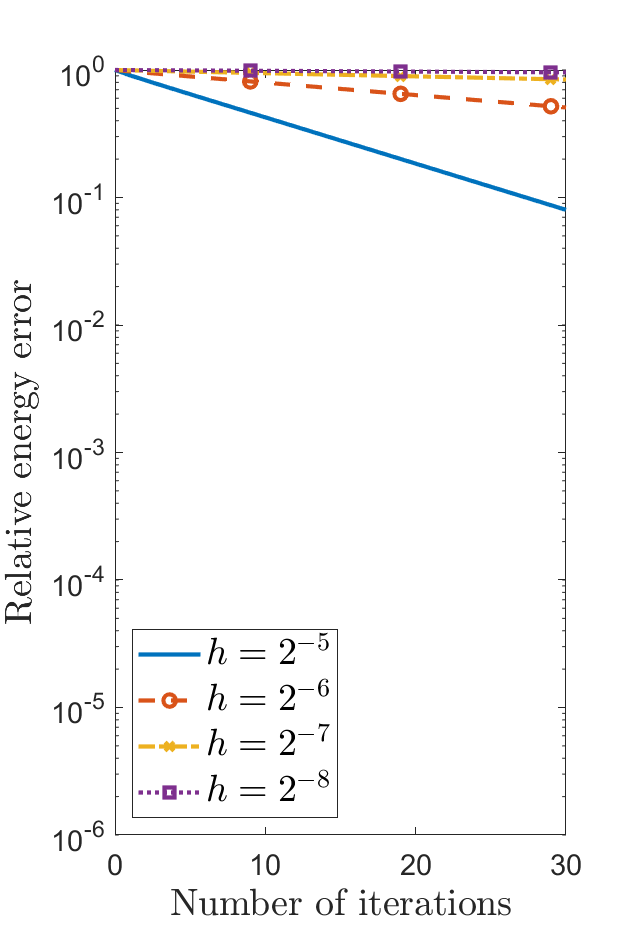}}
  \subfloat[][Two-level , $H/h = 2^3$, $\delta = 2h$]{\includegraphics[width=0.34\hsize]{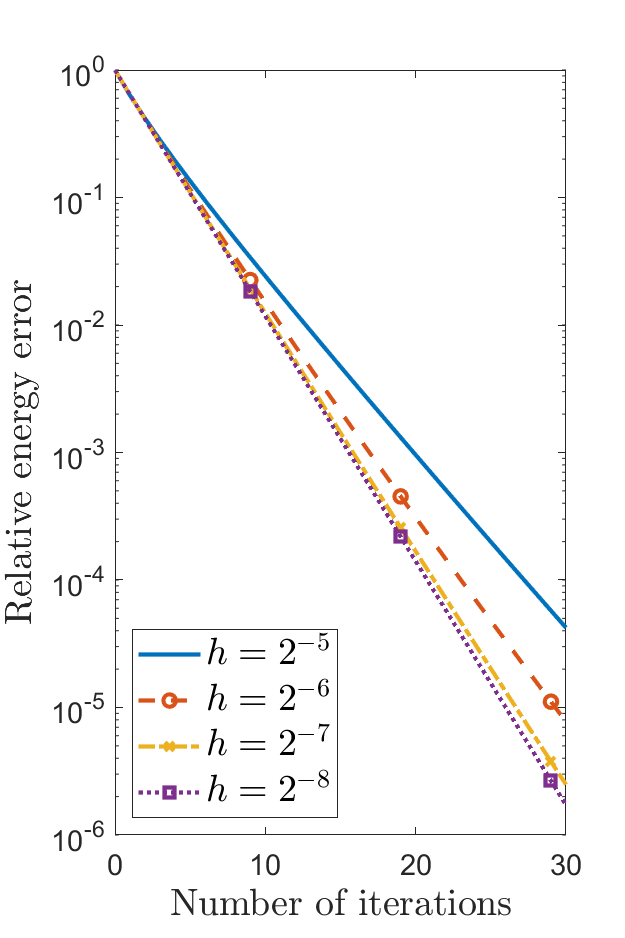}}
  \subfloat[][One-level , $H/h = 2^3$, $\delta = 4h$]{\includegraphics[width=0.34\hsize]{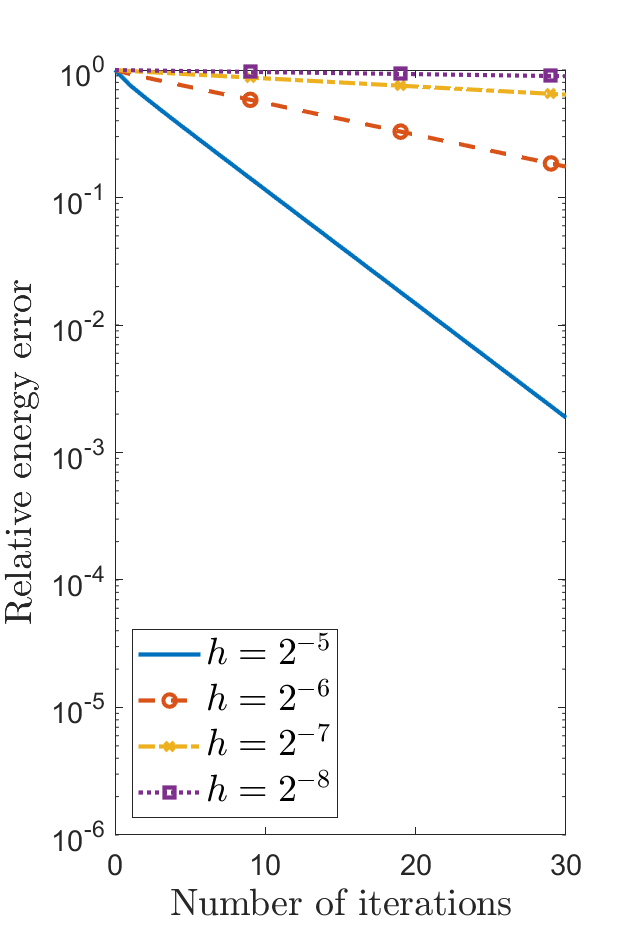}}
  \subfloat[][Two-level , $H/h = 2^3$, $\delta = 4h$]{\includegraphics[width=0.34\hsize]{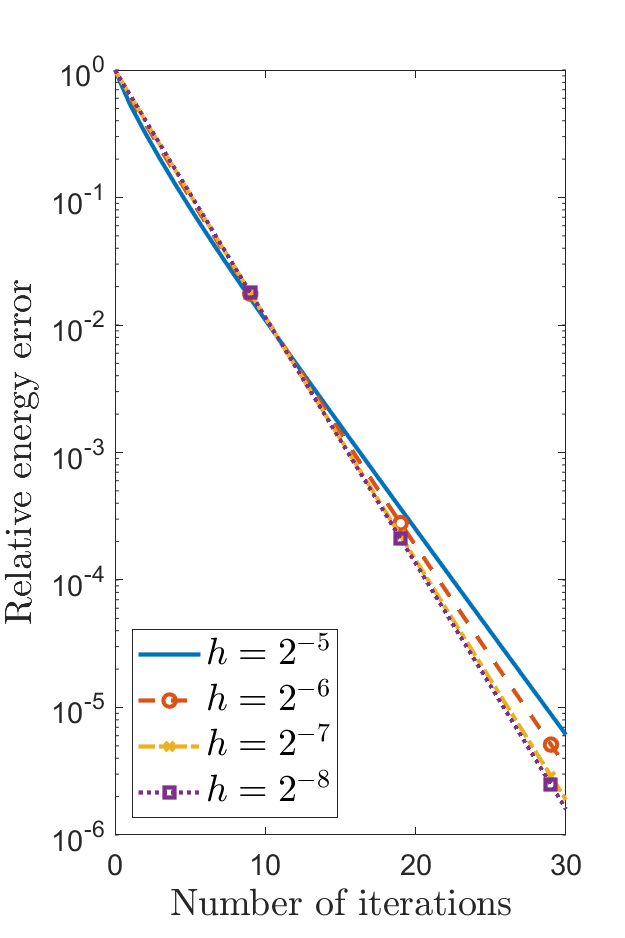}}
}
  \\
\resizebox{\textwidth}{!}{
  \subfloat[][One-level , $H/h = 2^4$, $\delta = 2h$]{\includegraphics[width=0.34\hsize]{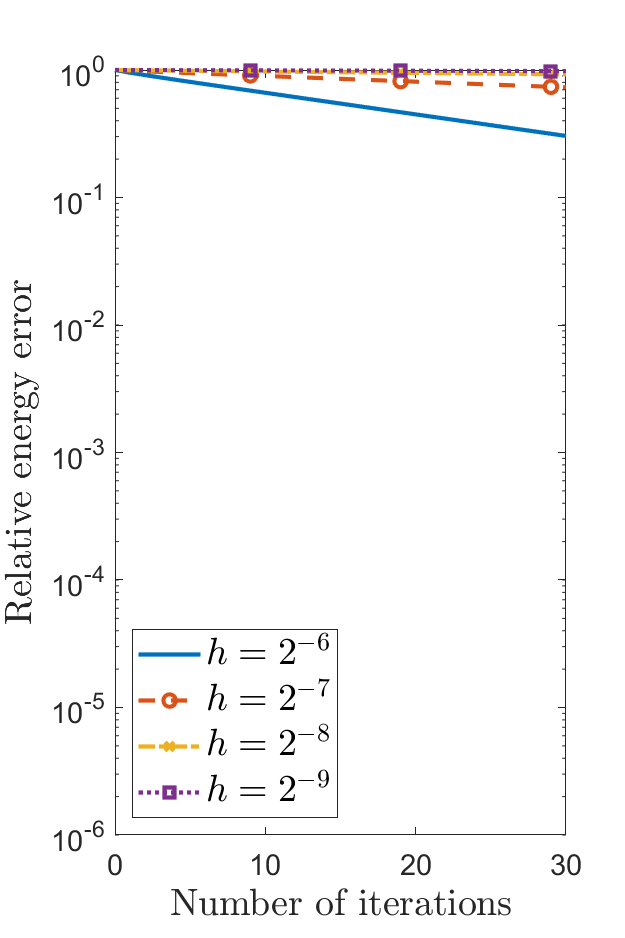}}
  \subfloat[][Two-level , $H/h = 2^4$, $\delta = 2h$]{\includegraphics[width=0.34\hsize]{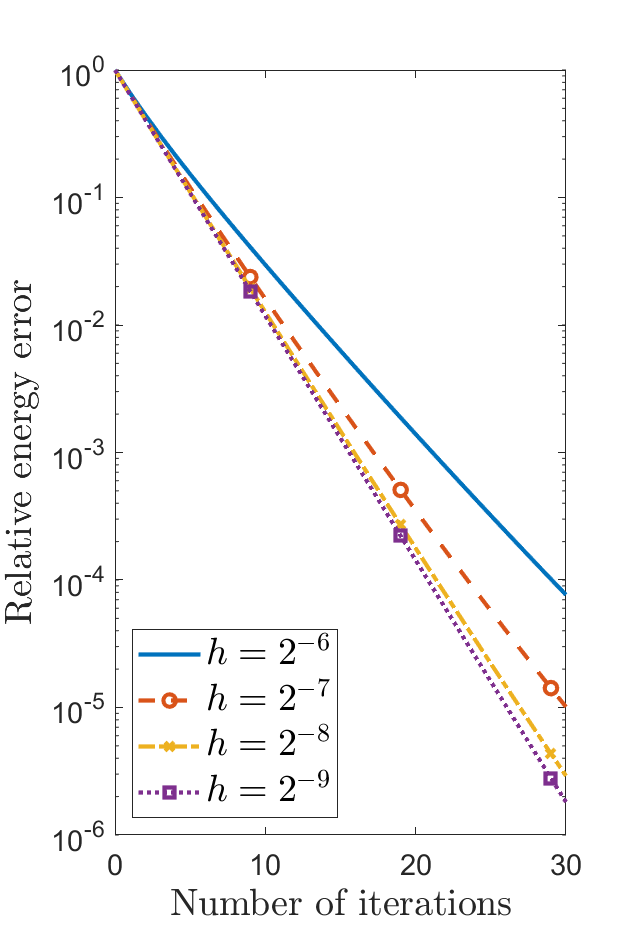}}
  \subfloat[][One-level , $H/h = 2^4$, $\delta = 2h$]{\includegraphics[width=0.34\hsize]{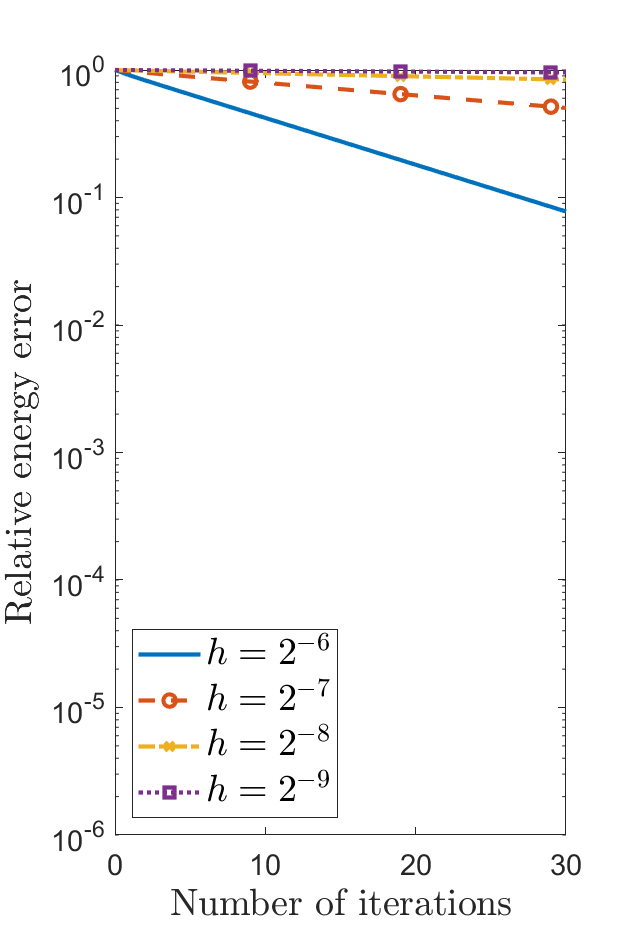}}
  \subfloat[][Two-level , $H/h = 2^4$, $\delta = 4h$]{\includegraphics[width=0.34\hsize]{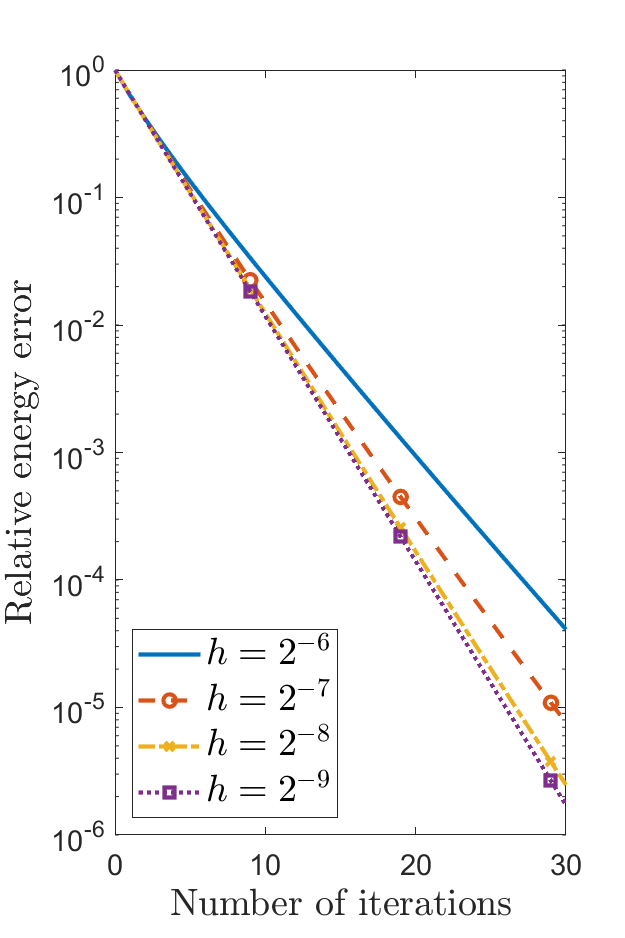}}
}
  \caption{Decay of the relative energy error $\frac{E_h (u^{(n)}) - E_h (u_h)}{|E_h (u_h)|}$ in additive Schwarz methods for the $L^1$-penalized problem~\eqref{L1}~($\alpha = 10$).}
  \label{Fig:L1}
\end{figure}

The convergence results of the one- and two-level additive Schwarz methods for~\eqref{L1} are summarized in \cref{Table:L1} and \cref{Fig:L1}.
From these results, we deduce similar conclusions as those observed in the cases of smooth problems such as~\eqref{monomial} and~\eqref{PB}.
In \cref{Table:L1}, the convergence rate is observed to be uniformly bounded with respect to $\alpha$.
In \cref{Fig:L1}, the two-level method exhibits numerical scalability, with the convergence rate remaining stable when $H/h$ and $H/\delta$ are fixed.

\section{Concluding remarks}
\label{Sec:Conclusion}
In this paper, we studied additive Schwarz methods for semilinear elliptic problems with convex energy functionals.
We established that nonlinearity of a semilinear elliptic problem can be localized, in the sense that the convergence rates of additive Schwarz methods are independent of the nonlinear term in the problem.
Hence, we can conclude that solving a semilinear is not much more difficult than solving a linear equation in view of domain decomposition methods, which is a similar conclusion to the one made in~\cite{Xu:1994}.

This paper raises several natural questions for future research.
The first question pertains to the extension of our analysis to quasilinear elliptic problems.
While existing results~\cite{Park:2020,TX:2002} have addressed the sublinear convergence of additive Schwarz methods for the $p$-Laplacian equation, which is an important example of degenerate quasilinear problems, a recent study~\cite{LP:2022} demonstrated that a sharper convergence estimate can be obtained by utilizing the notion of quasi-norm~\cite{BL:1993,EL:2005}.
Combining the idea of quasi-norm analysis with the results presented in paper is not a straightforward task.
The analysis provided in this paper relies on certain results that lack a quasi-norm generalization, including the trace theorem~\cite{DW:1994} and discrete Sobolev inequality~\cite{BX:1991}.
Consequently, extending the approach of this paper to quasilinear elliptic problems requires further investigation.

Another natural question is the extension to higher-order problems of the form
\begin{equation*}
\min_{u \in H_0^m (\Omega)} \intO \left( \frac{1}{2} | \nabla^m u |^2 + \phi (x, u) \right) \,dx,
\end{equation*}
where $m \in \mathbb{Z}_{\geq 2}$.
In our analysis, the positivity-preserving interpolation operator $J_H$ given in~\eqref{J_H} played a crucial role.
Meanwhile, constructing a positivity-preserving interpolation operator in higher-order finite element spaces has been regarded as a challenging problem for several decades.
In particular, in~\cite{NW:2001}, it was proven that linear positivity-preserving interpolation operators for higher-order finite elements with sufficient accuracy do not exist.
Very recently, a nonlinear positivity-preserving interpolation operator for the case $m = 2$ was proposed in~\cite{Park:2023}.
To extend our analysis to higher-order problems, the development of such positivity-preserving interpolation operators should be addressed first.

\appendix
\section{Proof of \texorpdfstring{\cref{Prop:semilinear_convex}}{Proposition 2.1}}
\label{App:Semilinear_convex}
In this appendix, we provide a proof of \cref{Prop:semilinear_convex}.
We first need the following elementary lemma, which can be easily proven by using the definition of convexity.

\begin{lemma}
\label{Lem:convex}
Let $f \colon \mathbb{R} \rightarrow \mathbb{R}$ be a convex function.
Then, for any $x \in \mathbb{R}$, the function $t \mapsto \frac{f(x+t) - f(x)}{t}$~($t \neq 0$) is increasing.
\end{lemma}

Using \cref{Lem:convex}, we can prove \cref{Prop:semilinear_convex} as follows.

\begin{proof}[Proof of \cref{Prop:semilinear_convex}]
Take any $u, v \in H_0^1 (\Omega)$.
We define
\begin{align*}
\Omega^+ &= \{ x \in \Omega : v(x) \geq 0 \}, \\
\Omega^- &= \{ x \in \Omega : v(x) < 0 \}.
\end{align*}
By~(ii) and \cref{Lem:convex}, for any $x \in \Omega$ and $t \in (-1, 1) \setminus \{ 0 \}$, we get
\begin{equation*} \begin{split}
\frac{\phi (x, u(x) + tv(x)) - \phi (x, u(x))}{t}
&\leq \begin{cases}
\phi (x, u(x) + v(x)) - \phi (x, u(x)), \quad & \text{ if } x \in \Omega^+ \\
\phi (x, u(x)) - \phi(x, u(x) - v(x)), \quad & \text{ if } x \in \Omega^-
\end{cases} \\
&\leq \left| \phi (x, u(x) + |v(x)|) - \phi(x, u(x)) \right|.
\end{split} \end{equation*}
Similarly, we have
\begin{equation*}
\frac{\phi (x, u(x) + tv(x)) - \phi (x, u(x))}{t}
\geq - \left| \phi (x, u(x) - |v(x)|) - \phi (x, u(x)) \right|.
\end{equation*}
By~(i), both the upper and lower bounds for $\frac{\phi (x, u(x) + t v(x)) - \phi (x, u(x))}{t}$ obtained above are in $L^1 (\Omega)$.
Hence, invoking the dominated convergence theorem yields
\begin{equation*} \begin{split}
\lim_{t \rightarrow 0} \frac{\intO \phi (x, u+tv) \,dx - \intO \phi (x, u) \,dx}{t}
&= \intO \lim_{t \rightarrow 0} \frac{\phi (x, u+tv) - \phi (x, u)}{t} \,dx \\
&\stackrel{\text{(iii)}}{=} \intO f(x, u) v \,dx.
\end{split} \end{equation*}
Since $f (\cdot, u) \in L^2 (\Omega)$ by~(i), we deduce that~\eqref{semilinear_convex} holds.
\end{proof}

\section{Proof of \texorpdfstring{\cref{Thm:FEM}}{Theorem 2.4}}
\label{App:FEM}
This appendix is devoted to proving \cref{Thm:FEM}, which provides an error estimate for the finite element discretization~\eqref{model_FEM} of the variational problem~\eqref{model}.
Throughout this appendix, we denote by $u$, $u_h$, and $\tu_h$ the solutions of the variational problems~\eqref{model},~\eqref{model_FEM}, and~\eqref{model_FEM_alt}, respectively.
Note that~\eqref{model_FEM_alt} employs the same energy functional as the continuous problem~\eqref{model}, whereas~\eqref{model_FEM} employs an approximated energy functional defined in terms of the trapezoidal rule.
Accordingly, the error analysis for $u_h$ shall be two-fold: an error analysis for $\tu_h$ and an analysis for the trapezoidal rule.

We first state an optimal a priori $H^1$-error estimate for $\tu_h$~\cite{Xu:1994}.

\begin{lemma}
\label{Lem:model_FEM_alt}
Let $u \in H_0^1 (\Omega)$ and $\tu_h \in S_h (\Omega)$ be the solutions of~\eqref{model} and~\eqref{model_FEM_alt}, respectively.
Suppose that the assumptions in \cref{Prop:semilinear_convex} and the following hold:
\begin{enumerate}[label=\emph{(\roman*)}]
\item $u \in H^2 (\Omega)$,
\item $v \mapsto f(\cdot, v)$ is Fr\'{e}chet differentiable.
\end{enumerate}
Then there exists a positive constant $C$ independent of $h$ such that
\begin{equation*}
\| u - u_h \|_{H^1 (\Omega)} \leq C h.
\end{equation*}
\end{lemma}
\begin{proof}
By \cref{Prop:semilinear_convex},~\eqref{model} is equivalent to~\eqref{model_weak}.
In addition, since the functional $v \mapsto \intO \phi (\cdot, v) \,dx$ is convex, the Poincar\'{e}--Friedrichs inequality~\cite{BS:2008} implies that $E$ is strongly convex in $H_0^1 (\Omega)$.
Hence, the map $- \Delta + ( f(\cdot, u))'$, which is well-defined by~(ii), is nonsingular.
Finally, combining~(i) and~\cite[equation~(2.4)]{Xu:1994} yields the desired result.
\end{proof}

Next, we state an $L^1 (\Omega)$-error estimate for the nodal interpolation operator $I_h$.

\begin{lemma}
\label{Lem:trapezoidal}
Let $I_h$ be the nodal interpolation operator associated with a quasi-uniform triangulation $\cT_h$ of $\Omega \subset \mathbb{R}^d$~($d = 2, 3$).
For $v \in C^0 (\overline{\Omega})$ such that $v|_T \in W^{2, \infty} (T)$ for all $T \in \cT_h$, there exists a positive constant $C$ independent of $v$ and $h$ such that
\begin{equation*}
\| v - I_h v \|_{L^1 (\Omega)} \leq C h^2 \max_{T \in \cT_h} | v |_{W^{2, \infty} (T)}.
\end{equation*}
\end{lemma}
\begin{proof}
Throughout this proof, let $C$ denote the generic constant independent of $v$ and $h$.
We write $w = v - I_h v$ and take any $T \in \cT_h$.
Let $x^0$ be a vertex of $T$ with the minimum solid angle, and let $e^i$, $1 \leq i \leq d$, denote a unit vector along each edge emanating from $x^0$.
Since $w = 0$ at the vertices of $T$, there exists a point $x^i \in \overline{T}$, $1 \leq i \leq d$, such that $\nabla_{e^i} w (x^i) = 0$, where $\nabla_{e^i}$ denotes the directional derivative with respect to $e^i$.
Note that, for any $x \in T$, the line segment $\overline{x^i x}$ joining $x^i$ and $x$ belongs to $T$ because $T$ is convex.
Hence, we have
\begin{equation*}
    | \nabla_{e^i} w (x) |
    = \left| \int_{\overline{x^i x}} \nabla \left( \nabla_{e^i} w \right) \cdot ds \right|
    \leq h | w |_{W^{2, \infty} (T)}
    = h |v|_{W^{2, \infty} (T)}.
\end{equation*}
It follows that
\begin{equation*}
    | w |_{W^{1, \infty} (T)} \leq C \sum_{i=1}^d | \nabla_{e^i} w |_{L^{\infty} (T)}
    \leq C h |v|_{W^{2, \infty} (T)},
\end{equation*}
where the minimum angle condition~\cite{BKK:2008} was used in the first inequality.
Since $w (x^0) = 0$, we obtain
\begin{equation}
    \label{Lem1:trapezoidal}
    \| w \|_{L^1 (T)}
    = \int_T \left| \int_{\overline{x^0 x}} \nabla w \cdot ds \right| \,dx
    \leq C h^{d+1} | w |_{W^{1, \infty} (T)}
    \leq C h^{d+2} | v |_{W^{2, \infty} (T)}.
\end{equation}
Finally, as $\cT_h$ consists of $\mathcal{O} (h^{-d})$ elements, summing~\eqref{Lem1:trapezoidal} over all $T \in \cT_h$ yields
\begin{equation*}
    \| w \|_{L^1 (\Omega)} = \sum_{T \in \cT_h} \| w \|_{L^1 (T)}
    \leq C h^2 \max_{T \in \cT_h} | v |_{W^{2,\infty}(T)},
\end{equation*}
which is our desired result.
\end{proof}

\begin{remark}
\label{Rem:Sobolev}
If we further assume that $d = 2$ and $v \in W^{2,1} (\Omega)$ in \cref{Lem:trapezoidal}, then one can deduce the following sharper error estimate~\cite[Theorem~4.4.20]{BS:2008} by invoking the Sobolev embedding theorem:
\begin{equation*}
    \| v - I_h v \|_{L^1 (\Omega)} \leq C h^2 |v |_{W^{2,1}(\Omega)}.
\end{equation*}
\end{remark}

Now, we are ready to prove \cref{Thm:FEM} by using \cref{Lem:model_FEM_alt,Lem:trapezoidal}.

\begin{proof}[Proof of \cref{Thm:FEM}]
Throughout this proof, let $C$ denote the generic constant independent of $h$.
By the strong convexity of $E$ and $E_h$ with respect to the $H^1 (\Omega)$-norm, we have
\begin{equation} \begin{split}
\label{Thm2:FEM}
\| \tu_h - u_h \|_{H^1 (\Omega)}^2 &\leq C \left( E_h (\tu_h) - E_h (u_h) \right), \\
\| \tu_h - u_h \|_{H^1 (\Omega)}^2 &\leq C \left( E(u_h) - E(\tu_h) \right),
\end{split} \end{equation}
respectively.
Hence, it follows by \cref{Lem:model_FEM_alt} and~\eqref{Thm2:FEM} that
\begin{equation*} \begin{split}
\| u - u_h \|_{H^1 (\Omega)}^2
&\leq C \left( \| u - \tu_h \|_{H^1 (\Omega)^2}^2 + \| \tu_h - u_h \|_{H^1 (\Omega)}^2 \right) \\
&\leq C \left[ h^2 + \left( E_h (\tu_h) - E_h (u_h) \right) + \left( E(u_h) - E(\tu_h) \right) \right] \\
&\leq C \left[ h^2 + \left| E(u_h) - E_h (u_h) \right| + \left| E(\tu_h) - E_h (\tu_h) \right| \right].
\end{split} \end{equation*}
Invoking \cref{Lem:trapezoidal}, we have
\begin{equation*}
 \left| E(u_h) - E_h (u_h) \right|
 \leq \| w - I_h w \|_{L^1 (\Omega )}
 \leq C h^2 \max_{T \in \cT_h} | w |_{W^{2, \infty} (T) },
\end{equation*}
where $w = \phi (\cdot, u_h)$.
The term $| E (\tu_h) - E_h (\tu_h) | $ can be estimated in the same manner.
In summary, we obtain~\eqref{Thm1:FEM}, which is our desired result.
\end{proof}

\section{Proof of \texorpdfstring{\cref{Thm:conv}}{Theorem 3.3}}
\label{App:conv}
In this appendix, we provide a proof of \cref{Thm:conv}.
First, in \cref{Lem:ASM}, we present the generalized additive Schwarz lemma for convex optimization, which was introduced in~\cite{Park:2020,Park:2022a}.
We note that the proof of \cref{Lem:ASM} can be established without major difficulty by closely following the arguments presented in~\cite[Lemma~3.4]{LP:2022}.

\begin{lemma}
\label{Lem:ASM}
Let $\{ u^{(n)} \}$ be the sequence generated by \cref{Alg:ASM}.
Then it satisfies
\begin{equation*}
    u^{(n+1)} \in \operatornamewithlimits{\arg\min}_{u \in V} \left\{ F_h (u^{(n)}) + \langle F_h' (u^{(n)}), u - u^{(n)} \rangle + M_{\tau} (u, u^{(n)}) \right\},
\end{equation*}
where the energy functional $E_h$ of~\eqref{model_FEM} admits the decomposition~\eqref{composite} and the functional $M_{\tau} \colon V \times V \rightarrow \mathbb{R}$ is given by
\begin{multline}
\label{M}
    M_{\tau} (u,v) = \tau \inf \Bigg\{ \sum_{k=1}^N \left( D_{F_h} (v +R_k^* w_k, v) + G_h (v + R_k^* w_k) \right) \\
    :
    u - v = \tau \sum_{k=1}^N R_k^* w_k,
    \text{ } w_k \in V_k \Bigg\}
    + (1 - \tau N ) G_h (v),
    \quad u,v \in V.
\end{multline}
\end{lemma}

Next, we introduce the following useful result, which is a rephrasing of~\cite[Lemma~4.6]{Park:2020} in a form suitable for our purpose.

\begin{lemma}
\label{Lem:M_bound}
Suppose that \cref{Ass:stable,Ass:convex} hold.
For $\tau \in (0, \tau_0]$, we have
\begin{multline*}
    D_{F_h} (u,v) + G_h (u)
    \leq M_{\tau} (u,v) \\
    \leq \frac{C_0^2}{2 \tau} | u - v|_{H^1 (\Omega)}^2 + \tau G_h \left( \frac{1}{\tau} u - \left( \frac{1}{\tau} - 1 \right) v \right) + (1 - \tau) G_h (v),
    \quad u, v \in V.
\end{multline*}
where the functional $M_{\tau}$ was given in~\eqref{M}.
\end{lemma}

Recall that the functional $F_h$ given in~\eqref{composite} is assumed to be $\mu$-strongly convex with respect to the seminorm $| \cdot |_{H^1 (\Omega)}$ for some $\mu \in [0, 1]$; see~\eqref{E_strong} for the definition of strong convexity.
Since $F_h$ is G\^{a}teaux differentiable, it satisfies
\begin{equation}
\label{F_strong}
D_{F_h} (u,v) \geq \frac{\mu}{2} | u - v |_{H^1 (\Omega)}^2,
\quad u, v \in V,
\end{equation}
where the Bregman distance $D_{F_h}$ was defined in~\eqref{Bregman}; see, e.g.,~\cite[Theorem~2.1.9]{Nesterov:2018}.

With the above ingredients, we are ready to prove \cref{Thm:conv} as follows.

\begin{proof}[Proof of \cref{Thm:conv}]
We take any $n \geq 0$.
For $u \in V$, we write
\begin{equation*}
\tilde{u} = \frac{1}{\tau} u - \left( \frac{1}{\tau} - 1 \right) u^{(n)}.
\end{equation*}
Then we have
\begin{equation}
\label{Thm1:conv}
\begin{split}
    E_h (u^{(n+1)})
    &\stackrel{\text{(i)}}{\leq} F_h (u^{(n)}) + \langle F_h '(u^{(n)}), u^{(n+1)} - u^{(n)} \rangle + M_{\tau} (u^{(n+1)}, u^{(n)}) \\
    &\stackrel{\text{(ii)}}{=} \min_{u \in V} \left\{ F_h (u^{(n)}) + \langle F_h' (u^{(n)}), u - u^{(n)} \rangle + M_{\tau} (u, u^{(n)}) \right\} \\
    &\stackrel{\text{(i)}}{\leq} \min_{u \in V} \bigg\{ F_h (u^{(n)}) + \tau \langle F_h' (u^{(n)}), \tilde{u} - u^{(n)} \rangle + \frac{\tau C_0^2}{ 2} | \tilde{u}  - u^{(n)} |_{H^1 (\Omega)}^2 \\
    &\quad\quad + \tau G_h ( \tilde{u} ) + (1 - \tau) G_h (u^{(n)}) \bigg\} \\
    &\stackrel{\eqref{F_strong}}{\leq} \min_{u \in V} \left\{ (1 - \tau) E_h (u^{(n)}) + \tau E_h (\tilde{u}) + \frac{\tau (C_0^2 - \mu)}{2} | \tilde{u} - u^{(n)}|_{H^1 (\Omega)}^2 \right\},
\end{split}
\end{equation}
where (i) and (ii) are established through Lemmas~\ref{Lem:M_bound} and~\ref{Lem:ASM}, respectively.
Substituting $u$ in the last line of~\eqref{Thm1:conv} with $tu_h + (1 - t) u^{(n)}$ for $t \in [0, \tau]$, we obtain
\begin{equation}
\label{Thm2:conv}
\begin{split}
    E_h (u^{(n+1)})
    &\leq \min_{t \in [0, \tau]} \bigg\{ (1 - \tau) E_h (u^{(n)}) + \tau E_h \left( \frac{t}{\tau}u_h + \left(1 - \frac{t}{\tau} \right) u^{(n)} \right) \\
    &\quad\quad + \frac{(C_0^2 - \mu) t^2}{2\tau} |u_h - u^{(n)}|_{H^1 (\Omega)}^2 \bigg\} \\
    &\stackrel{\eqref{E_strong}}{\leq} \min_{t \in [0, \tau]} \bigg\{ (1-t) E_h (u^{(n)}) + t E_h (u_h) \\
    &\quad\quad + \frac{(C_0^2 + 1 - \mu) t^2 - \tau t}{2 \tau} | u_h - u^{(n)} |_{H^1 (\Omega)}^2 \bigg\}.
\end{split}
\end{equation}

Next, we consider two cases $C_0^2 \geq \mu$ and $C_0^2 < \mu$ separately.
On the one hand, when $C_0^2 \geq \mu$, setting $t = \frac{\tau}{C_0^2 + 1 - \mu} \in [0,\tau]$ in the last line of~\eqref{Thm2:conv} yields
\begin{equation}
\label{Thm3:conv}
E_h (u^{(n+1)}) - E_h (u_h) \leq \left(1 - \frac{\tau}{C_0^2 + 1 - \mu} \right) (E_h (u^{(n)}) - E_h (u_h)).
\end{equation}
On the other hand, when $C_0^2 < \mu$, setting $t = \tau$ in the last line of~\eqref{Thm2:conv} yields
\begin{equation}
\label{Thm4:conv}
\begin{split}
E_h (u^{(n+1)}) - E_h (u_h) &\leq (1 - \tau) (E_h (u^{(n)}) - E_h (u_h)) + \frac{\tau (C_0^2 - \mu)}{2} |u_h - u^{(n)}|_{H^1 (\Omega)}^2 \\
&\leq (1 - \tau) (E_h (u^{(n)}) - E_h (u_h)).
\end{split}
\end{equation}
By combining~\eqref{Thm3:conv} and~\eqref{Thm4:conv}, we obtain the desired result.
\end{proof}

Recall that we have $\mu = 1$ in the setting~\eqref{FG}.
Hence, we deduce that \cref{Alg:ASM} satisfies the estimate~\eqref{Cor1:conv}.

\begin{remark}
\label{Rem:Park:2020}
    As mentioned earlier, we can alternatively apply the existing convergence theorem presented in~\cite[Theorem~4.8]{Park:2020} to estimate the linear convergence rate of \cref{Alg:ASM}.
    In this case, we obtain
    \begin{equation*}
        \frac{E_h (u^{(n+1)}) - E_h (u_h)}{E_h (u^{(n)}) - E_h (u_h)} \leq 1 - \frac{\tau}{2} \min \left\{ 1, \frac{1}{2C_0^2} \right\},
        \quad n \geq 0,
    \end{equation*}
    which is strictly weaker than \cref{Thm:conv} when $\mu = 1$.
    Hence, we conclude that \cref{Thm:conv} provides a better estimate than~\cite[Theorem~4.8]{Park:2020}.
    The reason is that \cref{Thm:conv} relies on stronger conditions, namely the strong convexity of both $E_h$ and $F_h$~(see~\eqref{E_strong} and~\eqref{F_strong}), whereas~\cite[Theorem~4.8]{Park:2020} is based solely on the sharpness condition~\eqref{E_sharp} of $E_h$.
\end{remark}

\bibliographystyle{siamplain}
\bibliography{refs_Schwarz_semilinear}

\begin{thebibliography}{10}

\bibitem{ABPR:1986}
{\sc E.~L. Allgower, K.~B{\"o}hmer, F.~Potra, and W.~Rheinboldt}, {\em A
  mesh-independence principle for operator equations and their
  discretizations}, SIAM J. Numer. Anal., 23 (1986), pp.~160--169.

\bibitem{Badea:2006}
{\sc L.~Badea}, {\em Convergence rate of a {S}chwarz multilevel method for the
  constrained minimization of nonquadratic functionals}, SIAM J. Numer. Anal.,
  44 (2006), pp.~449--477.

\bibitem{Badea:2010}
{\sc L.~Badea}, {\em One- and two-level additive methods for variational and
  quasi-variational inequalities of the second kind}, tech. report, Ser. Inst.
  Math. Rom. Acad., 2010.

\bibitem{BK:2012}
{\sc L.~Badea and R.~Krause}, {\em One-and two-level {S}chwarz methods for
  variational inequalities of the second kind and their application to
  frictional contact}, Numer. Math., 120 (2012), pp.~573--599.

\bibitem{BS:2011}
{\sc M.~Badiale and E.~Serra}, {\em Semilinear Elliptic Equations for
  Beginners}, Springer, London, 2011.

\bibitem{BL:1993}
{\sc J.~W. Barrett and W.~B. Liu}, {\em Finite element approximation of the
  $p$-{L}aplacian}, Math. Comp., 61 (1993), pp.~523--537.

\bibitem{BV:2004}
{\sc S.~Boyd and L.~Vandenberghe}, {\em Convex Optimization}, Cambridge
  University Press, Cambridge, 2004.

\bibitem{BX:1991}
{\sc J.~H. Bramble and J.~Xu}, {\em Some estimates for a weighted ${L}^2$
  projection}, Math. Comp., 56 (1991), pp.~463--476.

\bibitem{BKK:2008}
{\sc J.~Brandts, S.~Korotov, and M.~K{\v{r}}{\'\i}{\v{z}}ek}, {\em On the
  equivalence of regularity criteria for triangular and tetrahedral finite
  element partitions}, Comput. Math. Appl., 55 (2008), pp.~2227--2233.

\bibitem{BS:2008}
{\sc S.~C. Brenner and R.~Scott}, {\em The Mathematical Theory of Finite
  Element Methods}, Springer, New York, 2008.

\bibitem{CHX:2007}
{\sc L.~Chen, M.~J. Holst, and J.~Xu}, {\em The finite element approximation of
  the nonlinear {P}oisson--{B}oltzmann equation}, SIAM J. Numer. Anal., 45
  (2007), pp.~2298--2320.

\bibitem{DCPS:2012}
{\sc L.~B. Da~Veiga, D.~Cho, L.~F. Pavarino, and S.~Scacchi}, {\em Overlapping
  {S}chwarz methods for isogeometric analysis}, SIAM J. Numer. Anal., 50
  (2012), pp.~1394--1416.

\bibitem{DW:2009}
{\sc C.~R. Dohrmann and O.~B. Widlund}, {\em An overlapping {S}chwarz algorithm
  for almost incompressible elasticity}, SIAM J. Numer. Anal., 47 (2009),
  pp.~2897--2923.

\bibitem{DW:1994}
{\sc M.~Dryja and O.~B. Widlund}, {\em Domain decomposition algorithms with
  small overlap}, SIAM J. Sci. Comput., 15 (1994), pp.~604--620.

\bibitem{DFN:1992}
{\sc B.~A. Dubrovin, A.~T. Fomenko, and S.~P. Novikov}, {\em Modern
  Geometry---Methods and Applications. {P}art {I}}, Springer-Verlag, New York,
  second~ed., 1992.

\bibitem{EL:2005}
{\sc C.~Ebmeyer and W.~B. Liu}, {\em Quasi-norm interpolation error estimates
  for the piecewise linear finite element approximation of $p$-{L}aplacian
  problems}, Numer. Math., 100 (2005), pp.~233--258.

\bibitem{HS:2016}
{\sc C.~Heinemann and K.~Sturm}, {\em Shape optimization for a class of
  semilinear variational inequalities with applications to damage models}, SIAM
  J. Math. Anal., 48 (2016), pp.~3579--3617.

\bibitem{HKKR:2019}
{\sc A.~Heinlein, A.~Klawonn, J.~Knepper, and O.~Rheinbach}, {\em Adaptive
  {GDSW} coarse spaces for overlapping {S}chwarz methods in three dimensions},
  SIAM J. Sci. Comput., 41 (2019), pp.~A3045--A3072.

\bibitem{HKKRW:2022}
{\sc A.~Heinlein, A.~Klawonn, J.~Knepper, O.~Rheinbach, and O.~B. Widlund},
  {\em Adaptive {GDSW} coarse spaces of reduced dimension for overlapping
  {S}chwarz methods}, SIAM J. Sci. Comput., 44 (2022), pp.~A1176--A1204.

\bibitem{KJ:1990}
{\sc T.~Kerkhoven and J.~W. Jerome}, {\em ${L}_{\infty}$ stability of finite
  element approximations to elliptic gradient equations}, Numer. Math., 57
  (1990), pp.~561--575.

\bibitem{KF:2018}
{\sc D.~Kim and J.~A. Fessler}, {\em Adaptive restart of the optimized gradient
  method for convex optimization}, J. Optim. Theory Appl., 178 (2018),
  pp.~240--263.

\bibitem{LP:2021}
{\sc C.-O. Lee and J.~Park}, {\em A dual-primal finite element tearing and
  interconnecting method for nonlinear variational inequalities utilizing
  linear local problems}, Internat. J. Numer. Methods Engrg., 122 (2021),
  pp.~6455--6475.

\bibitem{LP:2022}
{\sc Y.-J. Lee and J.~Park}, {\em On the linear convergence of additive schwarz
  methods for the $p$-{L}aplacian}, arXiv preprint arXiv:2210.09183,  (2022).

\bibitem{LZHM:2008}
{\sc B.~Z. Lu, Y.~C. Zhou, M.~J. Holst, and J.~A. McCammon}, {\em Recent
  progress in numerical methods for the {P}oisson--{B}oltzmann equation in
  biophysical applications}, Commun. Comput. Phys., 3 (2008), pp.~973--1009.

\bibitem{Nesterov:2018}
{\sc Y.~Nesterov}, {\em Lectures on Convex Optimization}, Springer, Cham, 2018.

\bibitem{NW:2001}
{\sc R.~Nochetto and L.~Wahlbin}, {\em Positivity preserving finite element
  approximation}, Math. Comp., 71 (2001), pp.~1405--1419.

\bibitem{Oh:2013}
{\sc D.-S. Oh}, {\em An overlapping {S}chwarz algorithm for {R}aviart--{T}homas
  vector fields with discontinuous coefficients}, SIAM J. Numer. Anal., 51
  (2013), pp.~297--321.

\bibitem{OC:2015}
{\sc B.~O’Donoghue and E.~Candes}, {\em Adaptive restart for accelerated
  gradient schemes}, Found. Comput. Math., 15 (2015), pp.~715--732.

\bibitem{Park:2020}
{\sc J.~Park}, {\em Additive {S}chwarz methods for convex optimization as
  gradient methods}, SIAM J. Numer. Anal., 58 (2020), pp.~1495--1530.

\bibitem{Park:2021}
{\sc J.~Park}, {\em Accelerated additive {S}chwarz methods for convex
  optimization with adaptive restart}, J. Sci. Comput., 89 (2021), p.~Paper
  No.~58.

\bibitem{Park:2022a}
{\sc J.~Park}, {\em Additive {S}chwarz methods for convex optimization with
  backtracking}, Comput. Math. Appl., 113 (2022), pp.~332--344.

\bibitem{Park:2023}
{\sc J.~Park}, {\em Additive {S}chwarz methods for fourth-order variational
  inequalities}, arXiv preprint arXiv:2301.07260,  (2023).

\bibitem{Tai:2003}
{\sc X.-C. Tai}, {\em Rate of convergence for some constraint decomposition
  methods for nonlinear variational inequalities}, Numer. Math., 93 (2003),
  pp.~755--786.

\bibitem{Tai:2005}
{\sc X.-C. Tai}, {\em Nonlinear positive interpolation operators for analysis
  with multilevel grids}, in Domain Decomposition Methods in Science and
  Engineering, 2005, pp.~477--484.

\bibitem{TX:2002}
{\sc X.-C. Tai and J.~Xu}, {\em Global and uniform convergence of subspace
  correction methods for some convex optimization problems}, Math. Comp., 71
  (2002), pp.~105--124.

\bibitem{TW:2005}
{\sc A.~Toselli and O.~Widlund}, {\em Domain Decomposition Methods-Algorithms
  and Theory}, Springer, Berlin, 2005.

\bibitem{TSFO:2015}
{\sc G.~Tran, H.~Schaeffer, W.~M. Feldman, and S.~J. Osher}, {\em An ${L}^1$
  penalty method for general obstacle problems}, SIAM J. Appl. Math., 75
  (2015), pp.~1424--1444.

\bibitem{Xu:1994}
{\sc J.~Xu}, {\em A novel two-grid method for semilinear elliptic equations},
  SIAM J. Sci. Comput., 15 (1994), pp.~231--237.

\bibitem{Xu:1996}
{\sc J.~Xu}, {\em Two-grid discretization techniques for linear and nonlinear
  {PDE}s}, SIAM J. Numer. Anal., 33 (1996), pp.~1759--1777.

\bibitem{ZM:1974}
{\sc V.~E. Zakharov and S.~V. Manakov}, {\em On the complete integrability of a
  nonlinear {S}chr{\"o}dinger equation}, Theoret. Math. Phys., 19 (1974),
  pp.~551--559.

\end{thebibliography}
\end{document}